\documentclass[11pt]{amsart}
\usepackage[latin1]{inputenc}
\usepackage{amssymb}
\usepackage{pdfsync}
\usepackage[english]{babel}
\usepackage[left= 1.2 in, right= 1.2 in ,top= 1.3 in, bottom = 1.2 in]{geometry}

\usepackage[pdftex,pagebackref,colorlinks=true,urlcolor=blue,linkcolor=blue,citecolor=blue]{hyperref}
\usepackage[capitalize]{cleveref}
\usepackage{color}
\usepackage{amsmath}
\usepackage{amsfonts}
\usepackage{mathrsfs}
\usepackage{t1enc, graphicx}
\usepackage{verbatim}
\usepackage{bbm}
\usepackage{mathtools}

\usepackage[colorinlistoftodos]{todonotes}

\usepackage[normalem]{ulem}
\usepackage{enumitem}

\newcommand{\R}{\mathbb{R}}

\newcommand{\C}{\mathbb{C}}
\renewcommand{\d}{\,\mathrm{d}}
\newcommand{\T}{\mathbb{T}}
\newcommand{\Q}{\mathbb{Q}}
\newcommand{\orb}{\operatorname{orb}}
\newcommand{\Z}{\mathbb{Z}}
\newcommand{\N}{\mathbb{N}}
\renewcommand{\P}{\mathbb{P}}

\newcommand{\E}{\mathop{\mathbb{E}}}
\newtheorem{theorem}{Theorem}[section]
\newtheorem{proposition}[theorem]{Proposition}
\newtheorem{conjecture}[theorem]{Conjecture}
\newtheorem{lemma}[theorem]{Lemma}

\newtheorem{corollary}[theorem]{Corollary}
\theoremstyle{definition}
\newtheorem{definition}[theorem]{Definition}
\newtheorem*{definition*}{Definition}
\newtheorem{question}[theorem]{Question}
\theoremstyle{remark}

\newtheorem{example}[theorem]{Example}
\newtheorem{remark}[theorem]{Remark}

\newcommand{\edit}[3]{\color{#1}{#3}\color{black}\marginpar{\textcolor{#1}{[[#2]]}}}
\newcommand{\joel}[1]{\edit{blue!50}{JM}{#1}}

\newcommand{\wenbo}[1]{\edit{purple!90}{WS}{#1}}

\author{Sebasti\'an Donoso}
\author{Anh N. Le}
\author{Joel Moreira}
\author{Wenbo Sun}
\address{Departamento de Ingenier\'{\i}a Matem\'atica and Centro de Modelamiento Matem\'atico, CNRS IRL 2807,  Universidad de Chile, Beauchef 851, Santiago, Chile.}
\email{sdonoso@dim.uchile.cl}
\address{Department of Mathematics\\
	Ohio State University\\
	231 W. 18th Ave., Columbus, OH 43210}
\email{le.286@osu.edu  }

\address{Mathematics Institute\\ University of Warwick\\
Coventry, UK}
\email{joel.moreira@warwick.ac.uk}

\address{Department of Mathematics\\ Virginia Polytechnic Institute and State University\\ 225 Stanger Street,
	Blacksburg VA, 24061-1026, USA}
\email{swenbo@vt.edu}

\thanks{The first author was supported by ANID/Fondecyt/1200897 and  Centro de Modelamiento Matem\'atico (CMM), ACE210010 and FB210005, BASAL funds for centers of excellence from ANID-Chile.}

\title{Additive averages of multiplicative correlation sequences and applications}


\begin{document}
\maketitle

\begin{abstract}
We study sets of recurrence, in both measurable and topological settings, for actions of $(\mathbb{N},\times)$ and $(\mathbb{Q}^{>0},\times)$. 
In particular, we show that autocorrelation sequences of positive functions arising from multiplicative systems have positive additive averages.
We also give criteria for when sets of the form  $\{(an+b)^{\ell}/(cn+d)^{\ell}: n \in \mathbb{N}\}$ are sets of multiplicative recurrence, and consequently we recover two recent results in number theory regarding completely multiplicative functions and the Omega function.
    
\end{abstract}









\section{Introduction and results}

An equation of several variables is \emph{partition regular} if for any finite coloring of $\N$, there exists a solution to the equation with all variables having the same color. 
An old question of Erd\H{o}s and Graham \cite{Erdos_Graham80} asks whether the equation $x^2 + y^2 = z^2$ is partition regular. 
A partial answer was obtained in 2016, when the conjecture was confirmed in the case of two colors \cite{Heule_Kullmann_Marek16}. 
In general, the answer to Erd\H{o}s and Graham's question is still unknown even if we just require two of the variables to have the same color.  
\begin{conjecture}
    \label{ques:x^2-y^2}
    For any finite coloring of $\N$, there exist $x, y$ of the same color such that $x^2 + y^2$ (or $x^2-y^2$) is a perfect square.
\end{conjecture}
Frantzikinakis and Host proved in \cite{Frantzikinakis_Host_2017} that for any finite coloring of $\N$, there are $x, y$ of the same color such that $16x^2 + 9y^2$ is a perfect square. 
Expanding on these ideas, the fourth author established an analogue of \cref{ques:x^2-y^2} where $\N$ is replaced with the ring of integers of a larger number field (see \cite{Sun2018,Sun2019}). However the methods used there do not seem to apply to $\N$.

The above results were proved by first recasting the combinatorial problems into questions in ergodic theory, in particular about sets of return times in multiplicative measure preserving systems. For a semigroup $G$, a \emph{topological $G$-system} is a pair $(X, T)$ where $X$ is a compact metric space and $T$ is a $G$-action on $X$, i.e. for $g \in G$, $T_{g} \colon X \to X$ is a homeomorphism and for $g, h \in G$, $T_{g} \circ T_{h} = T_{gh}$. Let $\mathcal{B}$ be the Borel $\sigma$-algebra and $\mu$ be a probability measure on $X$ such that $\mu(T_{g}^{-1} A) = \mu(A)$ for all $A \in \mathcal{B}$. Then the tuple $(X, \mathcal{B}, \mu, T)$ is called a \emph{measure preserving $G$-system}. By a \emph{multiplicative  measure preserving system}, we mean a measure preserving $(\N, \times)$-system with $\times$ being the multiplication on $\N$.

For $x, y \in \N$, $x^2 + y^2$ is a perfect square if and only if $x = k(m^2 - n^2)$ and $y = k(2mn)$ for $k, m, n \in \N$ and $m > n$. Therefore, to answer the case $x^2+y^2$ of \cref{ques:x^2-y^2} using the approach in \cite{Frantzikinakis_Host_2017}, it suffices to solve the following conjecture:
\begin{conjecture}
\label{conj:pythagorean_ergodic}
    Let $(X, \mathcal{B}, \mu, T)$ be a measure preserving $(\N, \times)$-system and $A \in \mathcal{B}$ with $\mu(A) > 0$. There exist $m, n \in \N$ with $m > n$ such that
    \[
        \mu(T_{m^2 - n^2}^{-1} A \cap T_{2mn}^{-1} A) > 0. 
    \]
\end{conjecture}

The idea of using ergodic theoretical methods to solve combinatorial problems traces back to Furstenberg's proof of Szemer\'edi's Theorem \cite{Furstenberg77} in 1977. 
The novelty in Frantzikinakis and Host \cite{Frantzikinakis_Host_2017}'s approach is the usage of additive averages to study multiplicative measure preserving systems. 
This idea has the potential to address other unsolved problems in  ergodic Ramsey theory regarding partition regularity of polynomial equations and a primary goal of this paper is to begin a systematic study of this new tool and to obtain potential applications.
\subsection{Additive averages of multiplicative recurrence sequences}

Throughout this paper, for a finite set $E$ and function $f \colon E \to \C$, we use $\E_{n\in E}f(n)$ to denote the average $\frac{1}{|E|}\sum_{n\in E} f(n)$.
By the von Neumann ergodic theorem, for a multiplicative  measure preserving system $(X, \mathcal{B}, \mu, T)$, a set $A \in \mathcal{B}$ with $\mu(A) > 0$ and a multiplicative F{\o}lner sequence $(\Phi_N)_{N \in \N}$ in $\N$ (see \cref{sec:sets_mult_rec} for definitions), we have
\begin{equation}
\label{eq:von-neumann}
    \lim_{N \to \infty} \E_{n \in \Phi_N} \mu(A \cap T_{n}^{-1} A) > 0.
\end{equation}
On the other hand, it is not clear whether \eqref{eq:von-neumann} is still true if  $(\Phi_N)_{N \in \N}$ is replaced with the additive F{\o}lner sequence $([N])_{N \in \N}$. (Here $[N]$ denotes the set $\{1,\dots,N\}$.)

Our first result confirms that this is indeed the case. In fact, we obtain a more general theorem regarding multiple ergodic averages along subsemigroups of $(\N, \times)$.

\begin{theorem}
	\label{prop:positive_semigroup_intro}
	Let $(X, \mathcal{B}, \mu,T)$ be a measure preserving $(\N, \times)$-system and $A \in \mathcal{B}$ with $\mu(A)>0$. Let $G$ be a subsemigroup of $(\N, \times)$.
	Then for every $\ell \in \N$,
	\[
	    \limsup_{N\to\infty} \E_{n \in G \cap [N]} \mu(A \cap T_{n}^{-1} A \cap \ldots \cap T_{n^{\ell}}^{-1} A)>0.
	\]
\end{theorem}   

\cref{prop:positive_semigroup_intro} follows from the slightly more general \cref{prop:positive_semigroup} below.
Besides $\N$ itself, examples of multiplicative semigroups of $\N$ include $\{m^2 + Dn^2: m, n \in \N\}$ for some $D\in\N$, $\{a^n b^m: m, n \in \N\}$ and $\{m^a n^b q^c: m, n, q \in \N\}$ for $a, b, c \in \N$. 
We remark that there are two natural ways to take averages along these semigroups. 
One way is to enumerate its elements in increasing order as done in \cref{prop:positive_semigroup_intro}. 
Alternatively, one can exploit the fact that these semigroups admit parametrizations and take averages on the parameters. 
Our next result states that, also with this scheme, the averages of multiple recurrence sequences are positive. 
To formulate our result we need the notion of \emph{parametrized multiplicative function}, which describes certain parametrizations of multiplicative semigroups.
We postpone the formal definition to \cref{sec_ParametrizedMultiplicativeFunctions}, but here are some illustrative examples of commutative parametrized multiplicative functions:
\begin{itemize}
    \item $f\colon \N^2 \to \N$ defined by $f(n_1, n_2) = n_1^2 + D n_2^2$, for any $D\in\N$,
    \item $f\colon \N^2 \to \N$ defined by $f(n_1, n_2) = |n_1^2 + n_1 n_2 - n_2^2|$,
    \item $f\colon \N^3 \to \N$ defined by $f(n_1, n_2, n_3) = n_1^{a_1} n_2^{a_2} n_3^{a_3}$ for any $a_1,a_2,a_3\in\N$.
\end{itemize} 

\begin{theorem}
\label{thm:multipledeterminant_intro} 
	Let $(X, \mathcal{B}, \mu,T)$ be a measure preserving $(\N, \times)$-system and $A \in \mathcal{B}$ with $\mu(A)>0$. Let $f\colon \N^k\to\N$ be a commutative parametrized multiplicative function and $\ell\in\N$.
	Then
	\[
	    \liminf_{N\to\infty}\E_{n\in[N]^k} \mu\Big(A\cap T_{f(n)}^{-1} A \cap T_{f(n)^2}^{-1} A \cap \cdots \cap  T_{f(n)^{\ell}}^{-1} A\Big)>0.
	 \]
\end{theorem}
\cref{thm:multipledeterminant_intro}  follows from the more general \cref{thm:multipledeterminant} in \cref{sec:mix_ergodic_theory}.
As an illustrative special case of \cref{thm:multipledeterminant_intro} we obtain
\begin{proposition}\label{prop:multiple_squares_intro}
Let $(X, \mathcal{B}, \mu,T)$ be a measure preserving $(\N, \times)$-system and let $A \in \mathcal{B}$ with $\mu(A)>0$. 
Then
	\[
	    \liminf_{N\to\infty}\E_{m, n \in[N]} \mu(A\cap T_{m^2 + n^2}^{-1} A)>0.
	 \]
\end{proposition}
\subsection{Topological recurrence and applications to number theory}


A subset $R$ of the set of positive rational numbers $\Q^{>0}$ is called a \emph{set of (measurable) multiplicative recurrence} if for every $(\Q^{>0}, \times)$-system $(X, \mathcal{B}, \mu, T)$ and $A \in \mathcal{B}$ with $\mu(A) > 0$, there exists $r \in R$ for which $\mu(A \cap T_r^{-1} A) > 0$. By Furstenberg's correspondence principle, $R$ is a set of multiplicative recurrence if and only if for any $E \subset \N$ of positive upper multiplicative density (see \cref{sec:sets_mult_rec} for definitions), there exist $x, y \in E$ such that $x/y \in R$. If \cref{conj:pythagorean_ergodic} is true, then it follows that  $\{(m^2 - n^2)/(2mn): m, n \in \N, m > n\}$ is a set of multiplicative recurrence.

There is a weaker notion that suffices for our purpose of studying partition regularity called sets of topological multiplicative recurrence. 
A set $R \subset \Q^{>0}$ is called a \emph{set of topological multiplicative recurrence} if for every minimal topological $(\Q^{>0}, \times)$-system $(X, T)$ and $U \subset X$ open and non-empty, there exists $r \in R$ such that $U \cap T_r^{-1} U \neq \emptyset$. It turns out that $R$ is a set of topological multiplicative recurrence if and only if for any finite coloring of $\N$, there exist $x, y$ of the same color such that $x/y \in R$ (see \cref{lem:Q_and_N} below). 

In \cref{sec:sets_mult_rec} we show that if $\phi\colon (\N, +) \to (\Q^{>0}, \times)$ is a homomorphism and $S \subset \N$ is a set of topological (or measurable) additive recurrence, then $\phi(S)$ is a set of topological (or measurable) multiplicative recurrence. 
Using this fact, and drawing on the many known examples of sets of additive recurrence (cf. \cite{Furstenberg77, Kamae_France78, Sarkozy78}), one can exhibit examples of sets of multiplicative recurrence: $\{2^{P(n)}: n \in \N\}$ where $P \in \Z[x]$ satisfies $P(0) = 0$ and $\{2^{p-1}: p \text{ prime}\} $. 
However, all such examples are confined to a sparse semigroup of the form $\{a^n:n\in\N\}$, and there are hardly any other known examples of sets of multiplicative recurrence. Our next result produces a set of multiplicative recurrence which does not belong to the aforementioned class.

\begin{theorem}
\label{thm:n+1_n_intro}
    Let $a, c, \ell \in \N$, $b, d \in \Z$ and let $R:=\left\{ \left( \frac{an+b}{cn+d} \right)^{\ell}: n \in \N \right\}$.
    \begin{enumerate}
        \item If $a\neq c$, then $R$ is not a set  of topological multiplicative recurrence.
        \item If $a = c$, $b \neq d$ and there exists a prime $p$ such that $p | a$ but $p \nmid bd$, then $R$ is not a set  of topological multiplicative recurrence.
        \item If $a=c$ and either $a|b$ or $a|d$, then $R$ is a set of topological multiplicative recurrence. \footnote{In the special case $\ell=a=c=1$ and $b=d=0$, this result was also observed by Bergelson in \cite[Section 5]{Bergelson96}.}
    \end{enumerate} 
    
\end{theorem}

\cref{thm:n+1_n_intro} follows from combining Propositions \ref{cor:mobius_necessary}, \ref{cor:mobius_necessary2} and \ref{prop:mult_rec_mobius}.

A function $f\colon \N \to \C$ is called \emph{completely multiplicative} if $f(mn) = f(m) f(n)$ for $m, n \in \N$. Klurman and Mangerel \cite{Klurman_Mangerel_2018} proved that for any completely multiplicative function $f$ with $|f(n)| = 1$ for $n \in \N$, one has $\liminf_{n \to \infty} |f(n+1) - f(n)| = 0$. They remarked that their proof can be modified to treat the case in which the shift $1$ is replaced by any fixed $k \in \N$. As an application of \cref{thm:n+1_n_intro}, we give a strengthening of their result by allowing both arbitrary shifts and dilations.  

\begin{corollary}
\label{cor:completely_multiplicative_intro}
    Let $f: \N \to \C$ be a completely multiplicative function with $|f(n)| = 1$ for $n \in \N$. Then for all $a, k \in \N$, 
    \[
        \liminf_{n \to \infty} |f(an+k) - f(an)| =0.
    \]
\end{corollary}
The proof of \cref{cor:completely_multiplicative_intro} is given in \cref{sec:special_set_rec}, after the proof of \cref{prop:mult_rec_mobius}.
In the previous corollary, by setting $f=\xi^{\Omega(n)}$, where $\xi$ is a $q$-th root of $1$ and, for $n \in \N$, $\Omega(n)$ is the number of prime factors of $n$ counting with multiplicity, one immediately derives the following enhancement of \cite[Corollary 1.3]{Klurman_Mangerel_2018}.


\begin{corollary}
\label{cor:omega_function_intro}
    For any $a, q, k \in \N$, there exists infinitely many $n \in \N$ such that $\Omega(an + k)  \equiv \Omega(an) \bmod q$. 
\end{corollary}

Using (an appropriate version of) the Poincar\'e recurrence theorem it can be shown that every subsemigroup of $(\Q^{>0}, \times)$ is a set of multiplicative recurrence.
Therefore, to understand which algebraic sets are sets of multiplicative recurrence, one needs to know which contain a subsemigroup.
Our next theorem characterizes all polynomials of integer coefficients whose images contain subsemigroups of $(\Q^{>0}, \times)$
. 
\begin{theorem}
\label{thm:polynomial_semigroup}
    For $P \in \Z[x]$, the set $\{P(n): n \in \N\}$ contains an infinite subsemigroup of $(\Q^{>0}, \times)$ if and only if $P(x) =  (ax+b)^d$ for some $a, d \in \N$, $b \in \Z$ with $a|b(b-1)$.
\end{theorem}
\cref{thm:polynomial_semigroup} follows from the more general \cref{conj:polynomial_semigroup}.

Regarding linear polynomials, we show that containing subsemigroups is  not only a sufficient but also a necessary condition for their images to be sets of multiplicative recurrence. 

\begin{theorem}
\label{thm:linear_intro}
    For every $a, b \in \N$,  $\{a n + b: n \in \N\}$ is a set of multiplicative recurrence if and only if $a | b(b-1)$. 
\end{theorem}
\cref{thm:linear_intro} is implied by \cref{5.12} in \cref{sec:special_set_rec}.

Lastly, as previously mentioned, it is unknown whether the sets arising from Pythagorean triples are sets of multiplicative recurrence. However, we show that these sets do not contain subsemigroups of $(\Q^{>0}, \times)$. 
\begin{proposition}\label{prop:pythagorean_triples_semigroups}
    Neither of the sets
    \[
        \left\{\frac{m^2 + n^2}{2mn}: m, n \in \N\right\}
        \ \text{ and }\ \left\{\frac{2mn}{m^2 - n^2}: m, n \in \N, m>n\right\}
    \]
    contain an infinite subsemigroup of $(\Q^{>0}, \times)$.
\end{proposition}
\cref{prop:pythagorean_triples_semigroups} is proved in \cref{sec:candidates}, after \cref{lem:all_a_i_zero}.

\subsection*{Organization of the paper}
In \cref{sec:sets_mult_rec} we give the background regarding sets of multiplicative recurrence in both topological and measurable settings. \cref{sec:special_set_rec} is devoted to studying the topological recurrence properties of certain algebraic sets
and providing applications in number theory. In particular,  \cref{thm:n+1_n_intro}, \cref{cor:completely_multiplicative_intro}, \cref{cor:omega_function_intro} and \cref{thm:linear_intro} are proved in this section. In \cref{sec:candidates}, we analyze polynomial configurations and investigate when the image of $\N$ under a polynomial map contains a semigroup of $(\N,\times)$.
\cref{sec:mix_ergodic_theory} contains the ergodic theoretical point of view of this article where \cref{prop:positive_semigroup_intro} and \cref{thm:multipledeterminant_intro} will be proved.
Finally, in \cref{sec:finite_Pythagorean}, we study the problem of finding Pythagorean pairs and triples in dense subsets of finite fields. 

\subsection*{Acknowledgements}
We thank Florian Richter for help with the proof of \cref{prop_chromaticrecurrence} and thank Vitaly Bergelson for helpful conversations. We also thank the referee for suggestions that improve the organization and readability of the paper.

\section{Background}
\label{sec:sets_mult_rec}


In this section we collect basic facts about sets of recurrence for the semigroup $(\N,\times)$ of natural numbers and the semigroup $(\Q^{>0},\times)$ of positive rational numbers, both under multiplication. 
Most of the content is well known and presented here for completeness. 

\subsection*{Notations}
In this article, $\N=\{1,2.\ldots\}$ denotes the set of positive natural numbers, and $\Q^{>0}$ denotes the set of positive rational numbers.  
For a positive natural number $N$, we let $[N]$ denote the set $\{1,\ldots,N\}$.


 \subsection{Multiplicatively invariant density and Furstenberg's correspondence}

A \emph{F\o lner sequence} in a countable semigroup $G$ is a sequence $\Phi=(\Phi_N)_{N\in\N}$ of finite subsets of $G$ which are asymptotically invariant in the sense that for any $g\in G$, 
$$\lim_{N\to\infty}\frac{|g\Phi_N\cap\Phi_N|}{|\Phi_N|}=1,$$
where $g\Phi_N=\{gx:x\in\Phi_N\}$.
Given a F\o lner sequence $\Phi$ on a semigroup $G$, the \emph{upper density} of a set $A\subset G$ with respect to $\Phi$ is the quantity
$$\bar d_\Phi(A):=\limsup_{N\to\infty}\frac{|A\cap\Phi_N|}{|\Phi_N|}.$$
For a subset $A\subset G$ of a semigroup $G$ and an element $g\in G$, we write $S/g:=\{x\in G:xg\in A\}$.
It is easy to check from the definition that $\bar d_\Phi(A/g)=\bar d_\Phi(A)$ for any $g\in G$.

The \emph{upper Banach density} of $A$ is $d^*(A):=\sup_\Phi\bar d_\Phi(A)$, where the supremum is taken over all F\o lner sequences $\Phi$ on $G$.
This supremum is always achieved, so for each set $A$ there is a F\o lner sequence $\Phi$ such that $d^*(A)=\bar d_\Phi(A)$.

A discussion of F\o lner sequences and invariant densities in the case where $G = (\N, \times)$ is provided in \cite{Bergelson05}. As a concrete example one can take the sequence $\Phi_N:=\{x\in\N: x|N!\}$, but the exact choice will not be important in our discussion.
Note that any F\o lner sequence for $(\N,\times)$ is also a F\o lner sequence for $(\Q^{>0},\times)$.

Given a semigroup $G$, a \emph{measure preserving $G$-system} is a measure preserving system $(X,{\mathcal B},\mu,T)$ where $(X,{\mathcal B},\mu)$ is a probability space and $T=(T_g)_{g\in G}$ is an action of $G$ on $X$ by measure preserving transformations $T_g:X\to X$ which satisfy $T_{gh}=T_g\circ T_h$ for every $g,h\in G$.

The Furstenberg correspondence principle connects sets with positive upper density and measure preserving systems. 
Here is the version we will use. 
Given a subset $E$ of a semigroup $G$ and $n\in G$ we let $E/n$ denote $\{m\in G : mn \in E \}$.
\begin{theorem}[Furstenberg's correspondence principle for semigroups, cf. {\cite[Theorem 5.8 and Remark 5.9]{Bergelson06}}]\label{thm_correspondence}
    Let $G$ be a countable commutative semigroup, let $E\subset\N$ and let $\Phi$ be a F\o lner sequence in $G$.
    Then there exists a $G$-system $(X,{\mathcal B},\mu,T)$ and a set $A\in{\mathcal B}$ with $\mu(A)=\bar d_\Phi(E)$ and, for every $n_1,\dots,n_k\in G$,
    $$\bar d_\Phi\big(E\cap (E/n_1)\cap (E/n_2)\cap\cdots\cap (E/n_k)\big)\geq\mu\big(A\cap T_{n_1}^{-1}A\cap\cdots\cap T_{n_k}^{-1}A\big).$$
\end{theorem}

\subsection{Syndetic sets}

A subset $S\subset G$ is called  \emph{(right) syndetic} if there exists a finite set $F\subset G$ such that $S/F:=\bigcup_{x\in F}S/x= G$. The following proposition is well known and we include its proof for completeness.
\begin{proposition}\label{prop_syndetic}
	A set $S$ is syndetic if and only if it has positive upper density with respect to any F\o lner sequence. 
\end{proposition}
\begin{proof}
	If $S$ is syndetic and $\Phi$ is a F\o lner sequence, then take $F$ finite so that $S/F= G$. 
	We have $1=\bar d_\Phi(G)\leq\sum_{x\in F}\bar d_\Phi(S/x)=\bar d_\Phi(S)\cdot|F|$, so $\bar d_\Phi(S)>0$.
	
	Conversely, if $S$ is not syndetic and $\Phi$ is a F\o lner sequence, for every $N\in\N$ we have that $S/\Phi_N\neq G$.
	Take $x_N\in G\setminus (S/\Phi_N)$. 
	Then $\Psi_N:=x_N\Phi_N$ is disjoint from $S$. 
	On the other hand, the sequence $\Psi=(\Psi_N)_{N\in\N}$ is a F\o lner sequence, and so $\bar d_\Psi(S)=0$. 
\end{proof}

In order to prove \cref{prop:positive_semigroup_intro}, we will need the following theorem of Furstenberg and Katznelson \cite{Furstenberg_Katznelson85} (see also \cite[Page 9]{Bergelson05}).
Recall that a commutative semigroup is \emph{cancelative} if and only if it is a subsemigroup of a group.
\begin{theorem}[Furstenberg-Katznelson \cite{Furstenberg_Katznelson85}]\label{thm_Furstenberg_katznelson}
    Let $G$ be a cancelative commutative semigroup and let $T_1,\dots,T_{\ell}$ be commuting measure preserving actions of $G$ on the probability space $(X, \mathcal{B}, \mu)$.
    Then for every $A \in \mathcal{B}$ with $\mu(A) > 0$, there exist $\delta>0$ and a syndetic set $S\subset G$ such that for every $g\in S$,
    $$\mu(A\cap T_{1,g}^{-1}A\cap\cdots\cap T_{\ell,g}^{-1}A)>\delta.$$
\end{theorem}

When $\ell=1$, \cref{thm_Furstenberg_katznelson} becomes significantly simpler and is a version of Khintchine's recurrence theorem. 
It turns out that, in this case, it is possible to extend the scope of \cref{thm_Furstenberg_katznelson} by 
dropping the assumption that $G$ is cancelative and commutative.
This result is probably well known; we provide a short proof which is almost verbatim taken from the proof of \cite[Theorem 5.1]{Bergelson96}.

\begin{theorem}[Khintchine's  theorem]\label{thm:KhinchineGeneral}
	Let $G$ be a countable semigroup, let $(T_g)_{g\in G}$ be an action of $G$ by measure preserving transformations on a probability space $(X,{\mathcal B},\mu)$ and let $A \in \mathcal{B}$ with $\mu(A)>0$.
	Then for every $\epsilon>0$ the set
	$$R:=\big\{g\in G:\mu(A\cap T^{-1}_gA)>\mu^2(A)-\epsilon\big\}$$
	is syndetic.
\end{theorem}
\begin{proof}
	Suppose, for the sake of a contradiction that $R$ is not syndetic.
	Choose $g_1\in G\setminus R$ arbitrary. 
	Then, inductively, for each $m\geq1$ let $F_m=\big\{g_mg_{m-1}\cdots g_n:1\leq n\leq m\big\}$ and choose $g_{m+1}\in G\setminus RF_m^{-1}$.
	This is possible if $R$ is not syndetic.
	From the construction we have $F_m=g_mF_{m-1}\cup\{g_m\}$, and hence it follows that $F_m\cap R\subset\{g_m\}$ for every $m\in\N$.
	
	Let $N\in\N$ be large, to be determined later depending only on $\mu(A)$ and $\epsilon$ and let $f=\frac1N\sum_{n=1}^N1_A\circ T_{g_ng_{n-1}\cdots g_1}$.
	Clearly $\int_Xf\d\mu=\mu(A)$.
	By the Cauchy-Schwarz inequality,
	\begin{eqnarray*}
		\mu(A)^2
		&\leq&
		\int_Xf^2\d\mu
		=
		\frac1{N^2}\sum_{n,m=1}^N\mu\big(T_{g_ng_{n-1}\cdots g_1}^{-1}A\cap T_{g_mg_{m-1}\cdots g_1}^{-1}A\big)
		\\&=&
		\frac1{N^2}\left(\sum_{n=1}^N\mu(A)+2\sum_{m=2}^{N}\sum_{n=1}^{m-1}\mu\big(T_{g_ng_{n-1}\cdots g_1}^{-1}A\cap T_{g_mg_{m-1}\cdots g_1}^{-1}A\big)\right)  
		\\&=&
		\frac1{N^2}\left(N\mu(A)+2\sum_{m=2}^{N}\sum_{n=1}^{m-1}\mu\big(A\cap T_{g_mg_{m-1}\cdots g_{n+1}g_n}^{-1}A\big)\right)    
		\\&=&
		\frac1N\mu(A)+\frac2{N^2}\left(\sum_{m=2}^{N}\sum_{g\in F_m}\mu\big(A\cap T_g^{-1}A\big)\right)         
		\\&\leq&
		\frac1N\mu(A)+\frac2{N^2}\binom N2\big(\mu^2(A)-\epsilon\big)   
	\end{eqnarray*}
	which is a contradiction if $N$ is large enough.
\end{proof}

\subsection{Sets of multiplicative recurrence} \label{sec:sets_mult_recurrence}
Sets of recurrence for $(\Z,+)$-systems, and their connection to combinatorics were introduced by Furstenberg in his book \cite{Furstenberg81}.
Sets of recurrence for arbitrary countable semigroups were defined and studied in \cite{Bergelson_McCutcheon95}, and the specific case when the semigroup is $(\N,\times)$ in \cite{Bergelson05}.
\begin{definition}[Cf. {\cite[Definition 1.1]{Bergelson_McCutcheon95}}]
Let $G$ be a semigroup and let $R\subset G$.
We say that $R$ is a \emph{set of recurrence}, or a \emph{set of measurable recurrence} if for any measure preserving action $T=(T_g)_{g\in G}$ of $G$ on a probability space $(X,\mu)$ and any $A\subset X$ with $\mu(A)>0$ there exists $r\in R$ such that $\mu(A\cap T_r^{-1}A)>0$. 
\end{definition}

Examples of sets of recurrence include $G$ itself and any subsemigroup of $G$.
In particular, note that if $G$ admits an identity $e$, then the singleton $\{e\}$ is a set of recurrence.
More generally, the usual proof of Poincar\'e's recurrence theorem can be adapted to show that any $\Delta$-set (i.e., a set of the form $\{g_ig_{i+1}\cdots g_j:i\leq j\}$ for some sequence $(g_i)_{i\in\N}$ in $G$) is a set of recurrence.

\begin{proposition}\label{prop_homo}
Let $G,H$ be semigroups and let $\phi:G\to H$ be a homomorphism. 
If $R\subset G$ is a set of recurrence, then $\phi(R)$ is a set of recurrence in $H$.
\end{proposition}
\begin{proof}
 Composing any measure preserving action of $H$ with $\phi$ yields a measure preserving action of $G$ on the same space, and the conclusion follows.
\end{proof}

We are mainly interested in the semigroup $(\N,\times)$.
In view of the following lemma, we will also consider the group $(\Q^{>0},\times)$.
We denote by $d_\times^*(A)$ the Banach upper density of $A$ with respect to the multiplicative structure in either $(\N,\times)$ or $(\Q^{>0},\times)$.

The following proposition is well known among experts but we were unable to find a complete proof in the literature.
\begin{proposition}\label{lemma_multiplicativerecurrence}
Let $R\subset\Q^{>0}$. Then the followings are equivalent:
\begin{enumerate}
    \item $R$ is a set of recurrence, i.e., for any measure preserving action $T$ of $(\Q^{>0},\times)$ on a probability space $(X,\mu)$ and any $A\subset X$ with $\mu(A)>0$ there exists $r\in R$ such that $\mu(A\cap T_r^{-1}A)>0$. 
    \item For any measure preserving action $T$ of $(\N,\times)$ on a probability space $(X,\mu)$ and any $A\subset X$ with $\mu(A)>0$ there exists $a,b\in\N$ with $a/b\in R$ such that $\mu(T_a^{-1}A\cap T_b^{-1}A)>0$.
    \item For any $A\subset\Q^{>0}$ with $d_\times^*(A)>0$ there exists $r\in R$ such that $A\cap A/r\neq\emptyset$.
    \item For any $A\subset\Q^{>0}$ with $d_\times^*(A)>0$ there exists $r\in R$ such that $d_\times^*(A\cap A/r)>0$.
    \item For any $A\subset\N$ with $d_\times^*(A)>0$ there exist $a,b\in\N$ with $a/b\in R$ such that $A/a\cap A/b\neq\emptyset$.
    \item For any $A\subset\N$ with $d_\times^*(A)>0$ there exist $a,b\in\N$ with $a/b\in R$ such that $d_\times^*(A/a\cap A/b)>0$.
\end{enumerate}
\end{proposition}
\begin{proof}
  The implications (6)$\Rightarrow$(5) and (4)$\Rightarrow$(3) are trivial. The implications (1)$\Rightarrow$(4) and (2)$\Rightarrow$(6) follow directly from the Furstenberg correspondence principle, \cref{thm_correspondence}.
  
  The implications (3)$\Rightarrow$(1) and (5)$\Rightarrow$(2) have very similar proofs, based on an observation of Bergelson, so here we only give the proof of (5)$\Rightarrow$(2). Suppose we are given a measure preserving action $T$ of $(\N,\times)$ on a probability space $(X,\mu)$ and a set $A\subset X$ with $\mu(A)>0$. 
  Applying \cite[Theorem 3.14]{Bergelson05} we find a set $P\subset\N$ with $d_\times^*(P)>0$ and such that for any $a,b\in P$, $\mu(T_a^{-1}A\cap T_b^{-1}A)>0$.
  Using (5) we can then find $a,b\in P$ with $a/b\in R$, establishing (2).
  
  The implication (2)$\Rightarrow$(1) follows from the fact that any measure preserving action of $(\Q^{>0},\times)$ induces a natural measure preserving action of the subsemigroup $(\N,\times)$.
  Finally, the implication (1)$\Rightarrow$(2) follows from a routine inverse extension argument, which we provide for completeness.
  
  Let $(T_n)_{n\in\N}$ be a measure preserving action of $(\N,\times)$ on $(X,\mu)$. The invertible extension of the measure preserving system $(X,\mu,(T_n)_{n\in\N})$ is the system $(Y,\nu,(S_u)_{u\in\Q^{>0}})$ where $$Y=\left\{(x_u)_{u\in\Q^{>0}}\in X^{\Q^{>0}}:(\forall n\in\N)(\forall u\in\Q^{>0})x_{un}=T_nx_u\right\},$$ 
  $S_v(x_u)_{u\in\Q^{>0}}=(x_{uv})_{u\in\Q^{>0}}\subset X^{\Q^{>0}}$ and for any cylinder set $C=\prod_{u\in F}A_u\times\prod_{u\in\Q^{>0}\setminus F}X$ where $F\subset\Q^{>0}$ is finite and $A_u\subset X$ are arbitrary measurable sets we have $\nu(C)=\mu(\bigcap_{u\in F}T_{un}^{-1}A_u)$, where $n\in\N$ is any common denominator of all elements of $F$.
  It is routine to check that $(Y,\nu,(S_u)_{u\in\Q^{>0}})$ is indeed a measure preserving system, and that the projection defined by $(x_u)_{u\in\Q^{>0}}\mapsto x_{1}$ is a factor map between the $(\N,\times)$ systems $(Y,\nu,(S_u)_{u\in\N})$ and $(X,\mu,(T_n)_{n\in\N})$.
  On the other hand, the fact that (1) holds for the system $(Y,\nu,(S_u)_{u\in\Q^{>0}})$ implies that (2) holds for the system $(Y,\nu,(S_u)_{u\in\N})$.
  But the property (2) passes through to factors, so it also holds for the original system $(X,\mu,(T_n)_{n\in\N})$.
\end{proof}

A weaker notion of recurrence is \emph{topological recurrence}, which we now define.
Given a semigroup $G$, a \emph{topological $G$-system} is a pair $(X,T)$ where $X$ is a compact metric space and $T=(T_g)_{g\in G}$ is an action of $G$ on $X$ via continuous functions.
A topological $G$-system is \emph{minimal} if every orbit $\orb(x,T):=\{T_gx:g\in G\}$ is dense in $X$.
\begin{definition}
Let $G$ be a semigroup and let $R\subset G$.
We say that $R$ is a \emph{set of topological recurrence} if for any minimal $G$-system and any non-empty open set $U\subset X$ there exists $r\in R$ such that $U\cap T_r^{-1}U\neq\emptyset$. 
\end{definition}

One can not drop the minimality condition entirely in this definition, as otherwise no set $R$ would satisfy it. 
However, it is possible to slightly relax the minimality condition and still have an equivalent definition.
Recall that given a topological $G$-system $(X,T)$, a point $x\in X$ is called a \emph{minimal point} if its orbit closure $\overline{\orb(x,T)}$ is minimal.

\begin{proposition}\label{prop_minimalorbit}
    Let $G$ be a semigroup and let $R\subset G$.
    Then $R$ is a set of topological recurrence if and only if for any topological $G$-system $(X,T)$ with a dense set of minimal points, and any non-empty open set $U\subset X$ there exists $r\in R$ such that $U\cap T_r^{-1}U\neq\emptyset$. 
\end{proposition}
\begin{proof}
    Let $x\in U$ be a minimal point, let $Y=\overline{\orb(x,T)}$ and let $V:=Y\cap U\neq\emptyset$.
    Then $(Y,T)$ is a minimal system and $V$ a non-empty open subset of $Y$, so there exists $r\in R$ such that $V\cap T_r^{-1}V\neq\emptyset$, which implies that $U\cap T_r^{-1}U\neq\emptyset$. 
\end{proof}

If $G$ is amenable (in particular if $G$ is commutative) then in view of Bogolyubov-Krilov's theorem \cite{Bogoliouboff_Kryloff_1937}, every set of recurrence is a set of topological recurrence. 
We remark in passing that this is not true in general for non-amenable semigroups, and in fact Bergelson conjectures that it is never true for non-amenable groups (cf. \cite{Bergelson_2008} or \cite[Conjecture after Exercise 20]{Bergelson00}). 

One can also ask whether every set of topological recurrence is a set of measurable recurrence. 
When $G=(\N,+)$, a negative answer was given by Kriz \cite{Kriz87}.
We show below that we can use Kriz's example to answer the question in the negative for the semigroup $(\N,\times)$.

\begin{proposition}\label{prop_homo2}
Let $G,H$ be commutative semigroups and let $\phi:G\to H$ be a homomorphism. 
If $R\subset G$ is a set of topological recurrence, then $\phi(R)$ is a set of topological recurrence in $H$.
\end{proposition}
\begin{proof}
Let $T$ be a minimal continuous action of $H$ on the compact metric space $X$ and let $U\subset X$ be open and non-empty.
Let $S$ be the $G$-action on $X$ defined by $S_g=T_{\phi(g)}$, let $x\in U$, let $Y_0:=\overline{\{S_gx:g\in G\}}$ and let $Y\subset Y_0$ be a minimal $S$-subsystem.
Since $T$ is minimal, there exists $h\in H$ such that $V:=T_h^{-1}U\cap Y\neq\emptyset$. 

The set $V$ is a non-empty open subset of $Y$, so by the hypothesis there exists $g\in R$ such that $S_g^{-1}V\cap V\neq\emptyset$.
This implies that $S_g^{-1}T_h^{-1}U\cap T_h^{-1}U\neq\emptyset$, and hence $T_{\phi(g)}^{-1}U\cap U\neq\emptyset$ as well.
\end{proof}

\begin{corollary}
\label{cor:kriz}
There exists $R\subset\N$ which is a set of topological recurrence in $(\N,\times)$ but not a set of measurable recurrence in $(\N,\times)$.
\end{corollary}
\begin{proof}
 Let $R\subset\N$ be a set of additive topological recurrence which is not a set of additive recurrence.
 Let $\phi:\N\to\N$ be the map $\phi(n)=2^n$.
 Observe that $\phi:(\N,+)\to(\N,\times)$ is a homomorphism, so in view of \cref{prop_homo2} the set $\phi(R)$ is a set of multiplicative topological recurrence.
 
 Let $(X,\mu,T)$ be an additive measure preserving system, and $A\subset X$ satisfy $\mu(A)>0$ but $\mu(A\cap T^{-n}A)=0$ for all $n\in R$.
 Consider the multiplicative measure preserving action $S$ on $(X,\mu)$ given by $S_n=T^{a(n)}$, where $a(n)$ is the unique integer such that $n=2^{a(n)}b(n)$, with $b(n)$ odd.
 Since $a(\phi(n))=n$ for all $n\in\N$, it follows that $\mu(S_m^{-1}A\cap A)=0$ for all $m\in\phi(R)$, and hence that $\phi(R)$ is not a set of multiplicative recurrence.
\end{proof}

\begin{remark} 


We note that in general, shifts and dilations of sets of measurable multiplicative recurrence are not sets of measurable or even topological multiplicative recurrence. For example,   $\{4^n: n \in \N\}$ is a set of multiplicative recurrence since it is a semigroup. On the other hand, $\{4^n + 3: n \in \N\}$ is not since it is contained in $4 \N + 3$ and it is shown in \cref{thm:linear_intro} that the latter set is not a set of topological multiplicative recurrence. For dilations, the set $\{3 \cdot 4^n: n \in \N\}$ is not a set of multiplicative recurrence as seen in the following coloring: For $n \in \N$, write $n = 3^k q$ where $k \in \N \cup \{0\}$ and $q \in \N$ with $3 \nmid q$. Then color $c(n) = 0$ if $k$ is even and $c(n) = 1$ if $k$ is odd. 
\end{remark}




 
 

\begin{proposition}
\label{lem:Q_and_N}
Let $R\subset\Q^{>0}\setminus\{1\}$. Then the followings are equivalent:
\begin{enumerate}[ref=(\arabic*)]

\item \label{lem:Q_and_N item1} $R$ is a set of topological recurrence, i.e. for any minimal action $T$ of $(\Q^{>0},\times)$ on a compact metric space $X$ and any non-empty open set $U\subset X$ there exists $r\in R$ such that $U\cap T_{r}^{-1}U\neq\emptyset$.
  \item \label{lem:Q_and_N item2} For any topological $(\Q^{>0},\times)$-system $(X,T)$ with a dense set of minimal points and any non-empty open set $U\subset X$ there exist infinitely many $r\in R$ such that $U\cap T_{r}^{-1}U\neq\emptyset$.
    \item For any minimal action $T$ of $(\N^{>0},\times)$ on a compact metric space $X$ and any non-empty open set $U\subset X$ there exist infinitely many pairs $(a,b)\in\N^2$ with $a/b\in R$ such that $T_a^{-1}U\cap T_b^{-1}U\neq\emptyset$.
    \item For any finite partition $\Q^{>0}=C_1\cup\cdots\cup C_s$ there exists $i\in\{1,\dots,s\}$, $r\in R$ such that $C_i\cap C_i/r\neq\emptyset$.
    \item For any finite partition $\N=C_1\cup\cdots\cup C_s$ there exist $i\in\{1,\dots,s\}$, $a,b\in\N$ with $a/b\in R$ such that $C_i\cap C_i/r\neq\emptyset$.
    
\end{enumerate}

Moreover, any of the conditions from \cref{lemma_multiplicativerecurrence} implies any of the conditions from this lemma.
\end{proposition}
\begin{proof}
  Similar to \cref{lemma_multiplicativerecurrence}, the proof \cref{lem:Q_and_N} is routine, with the exception of the implication (1)$\Rightarrow$(3), which we prove here.
  Given a minimal multiplicative topological system $(X,(T_n)_{n\in\N})$, let $(Y,(S_u)_{u\in\Q^{>0}})$ be its invertible extension, defined so that $$Y=\left\{(x_u)_{u\in\Q^{>0}}\in X^{\Q^{>0}}:(\forall n\in\N)(\forall u\in\Q^{>0})x_{un}=T_nx_u\right\}\subset X^{\Q^{>0}}$$ 
  and $S_v(x_u)_{u\in\Q^{>0}}=(x_{uv})_{u\in\Q^{>0}}$.
  Given a cylinder set $C=\prod_{u\in F}A_u\times\prod_{u\in\Q^{>0}\setminus F}X\subset X^{\Q^{>0}}$, where $F\subset\Q^{>0}$ is finite and $A_u\subset X$ are open sets, we have $C\cap Y\neq\emptyset$ if and only if $\bigcap_{u\in F}T_{un}^{-1}A_u\neq\emptyset$, where $n\in\N$ is any common denominator of all elements of $F$.
  For any such $C$ and any $x\in Y$, since $(X,T)$ is minimal the orbit of $x_{n^{-1}}$ is dense in $X$ under $T$ and hence for some $m\in\N$ we have $T_mx_{n^{-1}}\in\bigcap_{u\in F}T_{un}^{-1}A_u$ so that $T_{nmu}x_{n^{-1}}\in A_u$ for all $u\in F$. 
  Since $x\in Y$ and $nmu\in\N$, we have that $T_{nmu}x_{n^{-1}}=x_{mu}$, and hence $(S_mx)_u=x_{mu}\in A_u$ for all $u\in F$. 
  We conclude that $(Y,(S_m)_{m\in\N})$ is also a minimal $(\N,\times)$ system.
  Since the projection defined by $(x_u)_{u\in\Q^{>0}} \mapsto x_{1}$ is a factor map between $(Y,(S_m)_{m\in\N})$ and $(X,T)$, the property (3) passes through to factors, and the property (1) for $(Y,(S_u)_{u\in\Q^{>0}})$ implies property (3) for $(Y,(S_m)_{m\in\N})$, this finishes the proof.  
\end{proof}

We remark in passing that items (4) and (6) of \cref{lemma_multiplicativerecurrence} also have an analogue for topological recurrence, using the notion of piecewise syndetic set. 

\begin{corollary}\label{prop_recurrencepowers}
Let $\ell \in \N$. A set $R \subset \Q^{>0}$ is a set of topological multiplicative recurrence if and only if $R' = \{r^{\ell}: r \in R\}$ is a set of topological multiplicative recurrence. 
\end{corollary}

\begin{proof}
Suppose $R \subset \Q^{>0}$ is a set of topological multiplicative recurrence. In view of \cref{prop_homo2}, we have $R'$ is a set of topological multiplicative recurrence. 

Now let $R$ be such that $R' = \{r^{\ell}: r \in R\}$ is a set of topological multiplicative recurrence. Letting $\chi$ be a finite coloring of $\N$, we claim that there are $x,y\in\N$ with $\chi(x)=\chi(y)$ and $x/y\in R$.
Consider the coloring $\tilde\chi$ defined as follows: for any $x \in \N$, write $x=y_x^\ell z_x$ where $y_x, z_x \in \N$ and $z_x$ is not divisible by any $\ell$-th power. (Note that $y_x$ and $z_x$ are uniquely determined by $x$.) Then for such $x$, define $\tilde{\chi}(x)=\chi(y_x)$.
Since $R'$ is a set of topological multiplicative recurrence, there exist $a,b\in\N$ such that $\tilde\chi(a)=\tilde\chi(b)$ and $a/b=r^\ell$ for some $r\in R$.
By writing $a=y_a^\ell z_a$ and $b=y_b^\ell z_b$, we have that $r^\ell=a/b=(y_a/y_b)^\ell\cdot z_a/z_b$.
Since $z_a$ and $z_b$ are not divisible by any $\ell$-th power, it follows that $z_a=z_b$ and $y_a/y_b=r$. Moreover, $\chi(y_a)=\tilde\chi(a)=\tilde\chi(b)=\chi(y_b)$. Hence, $\chi(y_a) = \chi(y_b)$ and $y_a/y_b \in R$; so our proof finishes. 
\end{proof}

\section{Sets of multiplicative recurrence and applications}


\label{sec:special_set_rec}

\subsection{Rational functions}


The purpose of this section is to prove \cref{thm:n+1_n_intro}. 
We start with the following necessary criteria for a set to be a set of topological multiplicative recurrence. 

\begin{lemma}
    \label{prop:accumulation_not_recurrence}
    For any $\epsilon,M\in\R$ satisfying $0 < \epsilon < 1 < 1 + \epsilon< M$, the set $S := \left( (\epsilon, 1 - \epsilon) \cup (1 + \epsilon, M) \right) \cap \Q$ is not a set of topological multiplicative recurrence. In other words, if $S$ does not have $0, 1$ or $\infty$ as an accumulation point, then $S$ is not a set of topological multiplicative recurrence.
\end{lemma}
\begin{proof}
    Let $\epsilon$ and $M$ be as in the hypothesis of the lemma. Choose $c_1, c_2, a \in \R$ such that
    \[
        c_1 < 1 - \epsilon < 1 + \epsilon < c_2
    \]
    and 
    \[
        a > \max \{c_2/\epsilon, M/c_1\}.
    \]
    Then it follows that
    \[
        \bigcup_{k \in \Z} (c_1 a^k, c_2 a^k) \subset (0,\epsilon) \cup (1 - \epsilon, 1 + \epsilon) \cup (M, \infty). 
    \]
    Let $S^1 = \{z \in \C: |z| = 1\}$ and consider topological $(\Q^{>0}, \times)$-system $(S^1, T)$ defined as $T_q (e(x)) = e(\log_a q) e(x)$ for every $q \in \Q^{>0}$ and $x \in [0,1)$ (here $e(x) := e^{2 \pi i x}$). Let $\epsilon_1 = -\log_a c_1, \epsilon_2 = \log_a c_2$ and let $B = \{e(x): x \in (1 - \epsilon_1, 1 + \epsilon_2)\}$ be a neighborhood of $e(0)$ in $S^1$. We have $T_q e(0) \in B$ if any only if $\log_a q \in \bigcup_{k \in \Z} (k - \epsilon_1, k + \epsilon_2)$, or equivalently $q \in \bigcup_{k \in \Z} (a^{-\epsilon_1} a^k, a^{\epsilon_2} a^k)$. It follows that
    \[
        \{q \in \Q^{>0}: T_q e(0) \in B\} \subset \bigcup_{k \in \Z} (a^{-\epsilon_1} a^k, a^{\epsilon_2} a^k) \subset (0,\epsilon) \cup (1 - \epsilon, 1 + \epsilon) \cup (M, \infty).
    \]
    Since every set of topological multiplicative recurrence must intersect $\{q \in \Q^{>0}: T_q e(0) \in B\}$ we conclude that $R$ must intersect $(0,\epsilon) \cup (1 - \epsilon, 1 + \epsilon) \cup (M, \infty)$. 
    Therefore the set $S$ defined in \cref{prop:accumulation_not_recurrence} is not a set of topological multiplicative recurrence. 
\end{proof}

We now prove parts (1) and (2) of \cref{thm:n+1_n_intro}. 

\begin{proposition}
    [Part  (1) of \cref{thm:n+1_n_intro}]
\label{cor:mobius_necessary}
Let $a, c, \ell \in \N$, $b, d \in \Z$ with $a \neq c$.
Then $S=\{(an+b)^{\ell}/(cn+d)^{\ell}:n\in\N\}$ is not a set of topological multiplicative recurrence. 
\end{proposition}
\begin{proof}
    First note that $(an + b)^{\ell}/(cn + d)^{\ell} \to (a/c)^{\ell}$ as $n \to \infty$. Therefore, for any $\epsilon > 0$, the set $S \setminus ((a/c)^{\ell} - \epsilon, (a/c)^{\ell} + \epsilon)$ is finite. Moreover, since $a \neq c$, for sufficiently small $\epsilon$, we have $((a/c)^{\ell} - \epsilon, (a/c)^{\ell} + \epsilon)$ does not contain $0, 1$. Hence, $S$ does not have $0, 1$ or $\infty$ as an accumulation point, and by  \cref{prop:accumulation_not_recurrence}, the set $S$ is not a set of topological multiplicative recurrence.
\end{proof}

\begin{proposition}[Part  (2) of \cref{thm:n+1_n_intro}]
\label{cor:mobius_necessary2}
  Let $a, \ell \in \N$ and $b, d \in \Z$. 
  Suppose that $b \neq d$ and there exists a prime $p$ for which $p | a$ but $p \nmid bd$. Then $S = \{(an + b)^{\ell}/(an + d)^{\ell}: n \in \N\}$ is not a set of topological multiplicative recurrence.
\end{proposition}
\begin{proof}
Since $S$ is a subset of $S' = \{(pn + b)^{\ell}/(pn + d)^{\ell}: n \in \N\}$, it suffices to prove that $S'$ is not a set of topological multiplicative recurrence. In view of \cref{prop_recurrencepowers}, we need to prove that $S'' = \{(pn+b)/(pn + d): n \in \N\}$ is not a set of topological multiplicative recurrence. 

Because $b \neq d$, there exists $u \in \N$ such that $p^u$ does not divide $b - d$. 
Let $\chi$ be the finite coloring of $\N$ defined recursively as 
\[
    \chi(n) = \begin{cases} n \bmod p^u &\text{ if } p \nmid n \\
    \chi(n/p) &\text{ if } p | n.\end{cases}
\]
We shall show that $S''$ is not a set of recurrence for the coloring $\chi$. By contradiction, assume there exist monochromatic $x, y \in  \N$ such that
\begin{equation}
\label{eq:frac_x_y}
    \frac{x}{y} = \frac{pn + b}{pn + d}.
\end{equation}
Write $x = p^{k_x} x_0$ and $y = p^{k_y} y_0$ where $k_x, k_y \in \N \cup \{0\}$ and $p \nmid x_0 y_0$. By the definition of the coloring $\chi$, we have that $x_0 \equiv y_0 \bmod p^u$.  

From \eqref{eq:frac_x_y}, we have $x(pn+d) = y(pn + b)$. It follows that $p^{k_x} x_0 (pn + d) = p^{k_y} y_0 (pn + b)$. Since $pn + b$ and $pn + d$ are coprime to $p$, we deduce that $k_x = k_y$ and $x_0(pn + d) = y_0(pn + b)$. As $x_0 \equiv y_0 \bmod p^u$, it must be that $pn + d \equiv pn + b \bmod p^u$, or equivalently, $b \equiv d \bmod p^u$. This contradicts our assumption that $p^u$ does not divide $b - d$ and our proof finishes. 
\end{proof}


For the proof of part (3) of \cref{thm:n+1_n_intro}, we need a recurrence result in topological dynamical systems:

\begin{lemma} \label{lemma:topologicalsyndetic} 
Let $(X, T)$ be a topological $(\N, +)$-system where $X$ is a compact metric space with the metric $d$. Suppose that there exists $x \in X$ such that $\{T^n x: n \in \N\}$ is dense in $X$. Then for every $y \in X, \epsilon > 0$ and $N \in \N$, there exists $m_0 \in \N$ such that $d(T^{i + m_0} x, T^i y) < \epsilon$ for all $i \in [N]$.

\end{lemma}

\begin{proof}
For $i \in [N]$, the map $T^i: X \to X$ is continuous, hence uniformly continuous. Thus, there exists $\delta>0$ such that for $w, z \in X$, $d(w, z)<\delta$ implies $d(T^{i}w,T^{i}z)<\epsilon$ for all $i\in [N]$. Since $\{T^n x: n \in \N\}$ is dense in $X$, for any $y\in Y$ we can find $m_0 \in \N$ such that $d(T^{m_0} x,y) < \delta$. Then the choice of $\delta$ implies that $d(T^{i + m_0}x,T^{i}y)<\epsilon$ for all $i \in  [N]$.  
\end{proof}

Next we prove the following special case of part (3) of \cref{thm:n+1_n_intro}.

\begin{lemma}\label{prop_chromaticrecurrence}
For any $a\in \N$, the sets $\{(an+1)/(an):n\in\N\}$ and $\{(an - 1)/(an): n \in \N\}$ are sets of topological multiplicative recurrence.
\end{lemma}


\begin{proof}
We only prove the first set is a set of topological multiplicative recurrence as the proof for the second set is the same. Fix $r\in \N$ with $r\geq 1$.
Let $\N_0$ denote $\N \cup \{0\}$ and let $X_0 = [r]^{\N_0} = \{(w(n))_{n \in \N_0}: w(n) \in [r]\}$. Define a metric $d$ on $X_0$ as the following: For $w = (w(n))_{n \in \N_0}$ and $z = (z(n))_{n \in \N_0}$, 
\[
    d(w, z) \colon = \sum_{n \in \N_0} \frac{|w(n) - z(n)|}{2^n}.
\]
Under this metric, $X_0$ becomes a compact metric space and $d(w,z) < 1$ implies $w(0) = z(0)$. Let $T: X_0 \to X_0$ be the left shift which is defined as $T(w(n))_{n \in \N_0} = (w(n+1))_{n \in \N_0}$. 
Then $T$ is a continuous map. 

Let $c: \N \to [r]$ be a finite coloring of $\N$. We need to find $x,y\in\N$ of the same color and $n\in\N$ such that $y/x=(an+1)/an$.
This is equivalent to finding $\{x,x+k\}$ of the same color with $ak|x$ (indeed, writing $x=nak$ and $y=x+k$ we have $y/x=(an+1)/an$).
By regarding $c$ as an element of $X_0$, we define $X$ to be the closure of $\{T^n c: n \in \N_0\}$ in $X_0$. It follows that $(X, T)$ is a topological $(\N_0, +)$-system. Let $(Y,T)$ be a minimal subsystem of $(X, T)$. \footnote{The existence of minimal subsystems follows from an application of Zorn's lemma (cf. \cite[Theorem 1.4]{Auslander88}).}
For each color $j \in [r]$, let $U_j$ be the cylinder $U_j=\{x \in X_0: x(0) = j\}$. By minimality, the return times of points in $Y \cap U_j$ back to $Y \cap U_j$ are uniformly bounded, i.e. there exists $N_j \in \N$ such that for all $y \in Y \cap U_j$, there exists $n = n(y) \in [N_j]$ satisfying $T^n y \in U_j \cap Y$ (for example, see \cite[Chapter 1]{Auslander88}). Setting $N_0 = \max \{N_j: j \in [r]\}$, we have that for all $y \in Y$ and $n \in \N$, there exists $1 \leq n' \leq N_0$ for which $y(n) = y(n + n')$.



Fix $y_0 \in Y$. By \cref{lemma:topologicalsyndetic}, there exists $m_0 \in \N$ such that $d(T^{i+m_0} c, T^i y_0) < 1$ for all $i \in [a N_0!]$ (the number $a N_0!$ is chosen here because later we will use the fact that $a n' | a N_0!$ for all $1 \leq n' \leq N_0$). In particular, we have $c(m_0 + i) = y_0(i)$ for $1 \leq i \leq a N_0!$. Let $i_0 \in [a N_0!]$ satisfy $m_0 + i_0 \equiv 0 \bmod a N_0!.$ From previous paragraph, there exists $1 \leq n' \leq N_0$ such that $y_0(i_0) = y_0(i_0 + n')$. Hence it follows that $c(m_0 + i_0) = c(m_0 + i_0 + n')$. By writing $x = m_0 + i_0$ and $k = n'$, we have $c(x) = c(x+k)$ with $ak|x$, finishing our proof. 
\end{proof}

We are now ready to prove part (3) of \cref{thm:n+1_n_intro}.

\begin{proposition}[Part (3) of \cref{thm:n+1_n_intro}]\label{prop:mult_rec_mobius}
    For $a, b, d, \ell \in \Z$ with $a, \ell > 0$ and $a|b$ or $a|d$, 
    \[
        S = \left\{ \left(\frac{an + b}{an+d} \right)^{\ell}: n \in \N \right\}
    \]
    is a set of topological multiplicative recurrence.
\end{proposition}

\begin{proof}
    In view of \cref{prop_recurrencepowers}, it suffices to show that
    \[
        S' = \left\{ \frac{an + b}{an + d}: n \in \N \right\}
    \]
    is a set of topological multiplicative recurrence. Without loss of generality, assume $a|d$ and $d = a d_1$. 
    If $b = d$, then $S' = \{1\}$ which is trivially a set of multiplicative recurrence. If $b > d$, then we have
    \[
        S' = \left\{ \frac{a(n + d_1) +b - a d_1}{a(n + d_1)} : n \in \N \right\}
    \]
    which contains the set
    \[
        \left\{ \frac{a(n + d_1) +b - a d_1}{a(n + d_1)}: n +d_1\equiv0\bmod(b-ad_1)\right\}  
        = \left\{ \frac{am + 1}{am}: m \in \N \right\}.
    \]
    By \cref{prop_chromaticrecurrence}, the last set is a set of topological multiplicative recurrence, and hence so is $S'$.
    
    On the other hand, if $b < d$, then $S'$ contains the set
    \[
        \left\{ \frac{a(n + d_1) +b - a d_1}{a(n + d_1)}: n+d_1\equiv0\bmod(ad_1-b) \right\} = \left\{ \frac{am - 1}{am}: m \in \N \right\}
    \]
    which is a again a set of topological multiplicative recurrence. 
\end{proof}

Next we prove \cref{cor:completely_multiplicative_intro}. 
\begin{proof}[Proof of \cref{cor:completely_multiplicative_intro}]
Let $f: \N \to \C$ be a completely multiplicative function with $|f(n)| = 1$ for $n \in \N$ and let $a, k \in \N$.
We need to prove that, 
    \[
        \liminf_{n \to \infty} |f(an+k) - f(an)| =0.
    \]

    Let $\epsilon > 0$. Consider the topological $(\N, \times)$-system $(S^1, T)$ where $S^1 = \{z \in \C: |z| = 1\}$ and $T_n e^{2 \pi i x} = f(n) e^{2 \pi i x}$ for all $x \in [0,1)$ and $n \in \N$. 
    Let $A = \{e^{2 \pi i x} \colon x \in (-\epsilon/2, \epsilon/2)\} \subset S^1$. 
    Then by \cref{prop:mult_rec_mobius} and \cref{lem:Q_and_N}, there exist infinitely many $n \in \N$ such that $T_{(an+k)m} A \cap T_{anm} A \neq \emptyset$ for some $m \in \N$ (which may depend on $n$). 
    This implies $f((an+k)m) A \cap f(an m) A \neq \emptyset$, or equivalently, $f(m) f(an+k) A \cap f(m) f(an) A \neq \emptyset$. Since multiplication by $f(m)$ is an isometry, we have
    \[
        f(an + k) A \cap f(an) A \neq \emptyset.
    \]
    It follows that $|f(an+k) - f(an)| \leq \epsilon$ for infinitely many $n$. Since $\epsilon$ is arbitrary, we have $\liminf_{n \to \infty} |f(an+k) - f(an)| = 0$. 
\end{proof}




\begin{remark}\label{thm_multipledivisor}
\cref{thm:n+1_n_intro} implies that for any finite coloring of $\N$ and every $k\in\N$, there exist $a, n\in\N$ such that $an|x$ and the set 
\[
    \left \{ x, x\frac{an+k}{an} \right \}
\]
is monochromatic. It is natural to ask if, more generally, we can find $a, n \in \N$ for which $an|x$ and the set 
\[
    \left\{ x,x\frac{an+1}{an},x\frac{an+2}{an},\dots,x\frac{an+k}{an} \right\}
\]
is monochromatic
It turns out that the answer is negative, even in the case $a=1$ and $k=2$, as shown by the following example.

Let $c: \N \to \{1, 2\}$ be the $3$-Rado coloring: For $m \in \N$, write $m = 3^k q$ where $k \in \N \cup \{0\}$ and $3 \nmid q$. Then color $c(m) = 1$ if $q \equiv 1 \mod 3$ and $c(m) = 2$ if $q \equiv 2 \mod 3$.
Note that in every $3$ consecutive integers, there is a number congruent to $1\bmod3$ and another congruent to $2\bmod3$, therefore there are no $3$ consecutive numbers with the same color.
On the other hand, observe that $c(x) = c(y)$ if and only if $c(xz)=c(yz)$. Indeed, this is obvious if $z$ is a power of $3$; and if $z$ is co-prime to $3$ then $c(xz) \equiv c(x) z \bmod3$.
If the set 
\[
    \left\{ \frac{x}{n} n, \frac{x}{n} (n+1), \frac{x}{n} (n+2) \right\}
\] 
were monochromatic for some $x, n \in \N$ with $n|x$, so would be the set $\{n,n+1,n+2\}$, a contradiction.
\end{remark}

\subsection{Linear polynomials}

Let $(X, T)$ be an $(\N, \times)$-system. If $X$ is a finite set, we say that $(X, (T_n)_{n \in \N})$ is a \emph{finite $(\N, \times)$-system.} 
The following theorem classifies which sets of the form $a \N + b$ that are sets of multiplicative recurrence, and hence implies \cref{thm:linear_intro}. 

\begin{theorem}\label{5.12}
    Let $a, b \in \N$ and $S = \{a n + b: n \in \N\}$. The followings are equivalent:
    \begin{enumerate}[label=(\roman*)]
        \item $a | b(b-1)$.
        \item $S$ is a multiplicative semigroup.
        \item $S$ contains an infinite multiplicative semigroup.
        \item $S$ is a set of measurable recurrence for $(\N, \times)$-systems.
        \item $S$ is a set of topological recurrence for finite $(\N, \times)$-systems. 
    \end{enumerate}
\end{theorem}
\begin{proof}
\noindent    (i) $\Rightarrow$ (ii): If $a | b(b-1)$, then for any $ax + b, ay + b \in S$, we have 
    \[
        (ax + b)(ay + b) = a \left( axy + bx + by + \frac{b(b-1)}{a} \right) + b \in S. 
    \]
    
\noindent (ii) $\Rightarrow$ (iii): Obvious.
    
\noindent    (iii) $\Rightarrow$ (iv): This follows from the Poincar\'e Recurrence Theorem.
    
\noindent    (iv) $\Rightarrow$ (v): Obvious.
    
\noindent    (v) $\Rightarrow$ (i): 
We prove the contrapositive. 
Let $a,b$ be such that $a \nmid  b(b-1)$. 
It suffices to construct a finite $(\N, \times)$-system $(X,(T_n)_{n\in\N})$ such that whenever $n \in S$, the permutation $T_n$ fixes no element. Let $a'=a/(a,b)$.
Observe that $b \not \equiv 1\bmod a'$, for otherwise $b(b-1)$ would be a multiple of $a=(a,b)a'$. We need to consider separate cases:

Case 1: $(a',b)=1$. In this case every $n \in S$ satisfies $(n,a') = 1$ and $n \equiv b \not \equiv 1 \bmod a'$.
Let $X=\big\{ x\in\{0, 1, \dots, a'-1\}: (x,a') = 1\big\}$.
For $n\in \N$, let $T_n\colon X \to X$ be defined as follows.
Let $A$ be the set of all $q \in \N$ such that the all prime divisors of $q$ also divide $a'$.
Any natural number $n$ can be written uniquely as $n=u v=u(n)v(n)$ where $u \in A, v \in \N$ and $(v,a')=1$.
For $x \in X$, define $T_n x=x v(n) \bmod a'$.
Notice that $v(nm)=v(n)v(m)$, so $(T_n)_{n\in \N}$ is indeed an action of $(\N,\times)$ on $X$.
Since every $n\in S$ satisfies $(n, a')=1$ and $n \not\equiv 1\bmod a'$, $v(n) = n \not \equiv 1 \bmod a'$. Therefore, for $n \in S$, we have $T_n$ does not fix any point in $X$.
    
Case 2: $(a',b)>1$. Let $p$ be a prime which divides $(a',b)$ and let $k \in \N$ be such that $p^{k-1} | b$ but $p^{k} \nmid b$.
    Let $X=\Z_k$ and let $T_n\colon X \to X$, $n\in \N$ be defined as follows.
    For $n\in\N$, let $u(n)$ and $v(n)$ be the unique non-negative integers satisfying $n = p^{u(n)}v(n)$ and $p \nmid  v(n)$.
    For $x \in X$, define $T_n x:= x + u(n) \bmod k$.
    Since $u(nm)=u(n) + u(m)$, $(T_n)_{n\in\N}$ indeed defines an action of $(\N,\times)$ on $X$. Because $(a', b/(a,b)) = 1$ and $p | a'$, it follows that $p \nmid b/(a,b)$. On the other hand, $p^{k-1} | b$. Therefore, $p^{k-1} | (a,b)$, and hence $p^k | a = (a,b) a'$. Thus for any $n \in S$, we have $u(n) = k-1$, which implies that $T_n x = x - 1 \bmod k$ for every $x \in X$. In particular, $T_n$ does not fix any element of $X$.
\end{proof}

A polynomial $P \in \Z[x]$ is called \emph{divisible} if for any $q \in \N$, there exists $n \in \N$ such that $q | P(n)$. As proven in \cite{Kamae_France78}, the set $\{P(n): n \in \N\}$ is a set of additive topological (and measurable) recurrence if and only if $P$ is divisible.

\begin{corollary}
\label{cor:divisible}
    Let $P \in \Z[x]$. If $S = \{P(n): n \in \N\}$ is a set of topological multiplicative  recurrence, then $R = \{P(n)(P(n) - 1): n \in \N\}$ is a set of measurable additive recurrence, i.e. the polynomial $P(P-1)$ is divisible.
\end{corollary}
\begin{proof}
    Assume $P(P-1)$ is not divisible. Then there exists $a \in \N$ such that $a \nmid P(n)(P(n)-1)$ for all $n \in \N$. 
    Hence a 
    \begin{equation}\label{eq_notasetofrecurrence}
        \{P(n): n \in \N\} \subset \bigcup_{\substack{0 \leq b \leq a - 1 \\ a \nmid b(b-1)}} a \N + b.
    \end{equation}
    
    In view of \cref{5.12}, for each $b\in\{0,\dots,a-1\}$ such that $a\nmid b(b-1)$, the set $a\N+b$ is not a set of multiplicative recurrence. Invoking \cite[Theorem 2.7 (b)]{Bergelson_McCutcheon95} it follows that the union in the right hand side of \eqref{eq_notasetofrecurrence} can not be a set of recurrence.
\end{proof}
\begin{example}
  As an application of \cref{cor:divisible}, the set $\{n^{2}+3: n \in \N\}$ is not a set of topological multiplicative recurrence. This is because the polynomial $Q(n)=n^{4}+5n^{2}+6$ is not divisible, as $5 \nmid Q(n)$ for any $n\in\N$.
\end{example}

Any polynomial $P \in \Z[x]$ with $P(0) = 0$ is divisible, so in that case the set $S = \{P(n): n \in \N\}$ is a set of measurable additive recurrence. However, this set is not a set of multiplicative recurrence in general as the following example shows:

\begin{example}
\label{ex:2n^2}
    Let $p$ be a prime. The set $S = \{pn^2: n \in \N\}$ is not a set of topological multiplicative recurrence. Indeed, let $E = \{p^{2k} \ell: k \in \N \cup \{0\}, \ell \in \N, p \nmid \ell\}$ and $F = \{p^{2k + 1} \ell: k \in \N \cup \{0\}, \ell \in \N, p \nmid \ell\}$. Then $E \cup F = \N$. For every $x, y \in E$ or $x, y \in F$, $x/y$ has the form  $p^{2k} \ell_1/\ell_2$ for some $k \in \Z$ and $\ell_1, \ell_2 \in \N$ with $p \nmid \ell_1 \ell_2$. It follows that $x/y \not \in S$, hence $S$ is not a set of multiplicative recurrence.
\end{example}

\cref{ex:2n^2} can also be used to show that the converse to \cref{cor:divisible} is false in general. 
Indeed, 
take $P(n) = pn^2$. Then $P(n)(P(n) - 1) = pn^2(pn^2 - 1)$ is divisible, but $\{P(n): n \in \N\}$ is not set of topological multiplicative recurrence.


\section{Algebraic sets that contain multiplicative semigroups}
\label{sec:candidates}

As mentioned in \cref{sec:sets_mult_rec}, every subsemigroup of $(\N,\times)$ is a set of multiplicative recurrence. 
With the ultimate goal of understanding which algebraically defined sets are sets of multiplicative recurrence, it is natural to also study sets which are (or contain) a semigroup. 
In this section we investigate this question for sets that are the image of a polynomial or the image of a rational function.

\subsection{Polynomial sets that contain multiplicative semigroups}

In this section we prove \cref{thm:polynomial_semigroup}. We first need some lemmas. 

\begin{lemma}
\label{prop:d-1}
    Let $d \geq 2$ and $P(x) = a_d x^d + a_{d-2} x^{d-2} + \ldots + a_1 x + a_0\in\Q[x]$ such that at least one of $a_0, a_1, \ldots, a_{d-2}$ is non-zero and $a_d > 0$. Then $S = \{P(n): n \in \N\}$ does not contain any infinite subsemigroup of $(\Q^{>0}, \times)$.
\end{lemma}
\begin{proof}
    First we show that, for any $D\in\N\setminus\{1\}$, there are only finitely many $n \in \N$ such that $D^d P(n) = P(Dn)$. 
    Indeed, otherwise, the polynomial $D^d P(x) - P(Dx)$ would have infinitely many roots, and hence $D^d P(x) = P(Dx)$ for all $x \in \R$. 
    It would follow that $P(D^k) = P(1) (D^k)^d $ for all $k \in \N$. This implies that the equation $P(x) = P(1) x^d$ has infinitely many solutions. So $P(x) = P(1) x^d$ for all $x \in \R$. This is contradicts our assumption that at least one of $a_0, a_1, \ldots, a_{d-2}$ is nonzero.

    Next we assume, for the sake of a contradiction, that there exists $D \in \N\setminus\{1\}$ such that $\{D^n: n \in \N\} \subset S$. Then there exist infinitely many $m > n \in \N$ such that
    \begin{equation}\label{eq:compare1}
        D^d P(n) = P(m).
    \end{equation}
    Using the first paragraph of the proof, there are infinitely many solutions to \eqref{eq:compare1} with $m\neq Dn$.
    Next, observe that \eqref{eq:compare1} can be written as
    \begin{equation}
    \label{eq:compare}
        a_d (m^d - (D n)^d) = \sum_{j=0}^{d-2} a_j (D^d n^j - m^j).
    \end{equation}
    Note that for all $x , y \in \N$ with $x \neq y$, $|x^d - y^d| > (\max\{x, y\})^{d-1}$. Hence if $m \neq Dn$ for some $m, n \in \N$, then $|m^d - (Dn)^d| > (\max\{m, Dn\})^{d-1}$. On the other hand,
    \[
        \left| \sum_{j=0}^{d-2} a_j (D^d n^j - m^j) \right| \leq 2 (d-1) \max_{0 \leq j \leq d-2} \{|a_j|\} (\max\{m, Dn\})^{d-2}.
    \]
    Therefore, from \eqref{eq:compare} we  derive that 
    \[
        a_d \max\{m, Dn\}^{d-1} < 2 (d-1) \max_{0 \leq j \leq d-2} \{|a_j|\} \max\{m,Dn\}^{d-2}
    \]
    for infinitely many $m, n \in \N$. This is impossible and hence it yields the desired contradiction.
\end{proof}

\begin{lemma}
\label{lem:factorP}
Let $P\in\Z[x]$. Then there exist $a,b\in\Z$ and $Q\in\Q[x]$ such that $P(x)=Q(ax+b)$ and writing $Q(x)=a_dx^d+a_{d-1}x^{d-1}+\cdots+a_1x+a_0$ we have $a_{d-1}=0$.
If the leading coefficient of $P$ is positive, then $a>0$.
\end{lemma}
\begin{proof}
 Let $P(x)=b_dx^d+\cdots+b_0$ and factor it over $\C$ as $b_d\prod_{i=1}^d(x-\alpha_i)$, where $\alpha_1,\dots\alpha_d\in\C$ are the roots of $P$ counted with multiplicity.
 It follows that $\sum_{i=1}^d\alpha_i= -b_{d-1}/b_d$.
 Let $a=db_d$,  $b=b_{d-1}$ and define $Q(y):=P((y-b)/a)$.
 Observe that $Q(y)=0$ if and only if $y=a\alpha_i+b$ for some $1\leq i\leq d$.
 Therefore the sum of the roots of $Q$ is $0$, which implies that $a_{d-1}=0$.
\end{proof}

We are ready to prove \cref{thm:polynomial_semigroup}. In fact we prove the following more general result.

\begin{theorem}
\label{conj:polynomial_semigroup}
    Let $P \in \Z[x]$ and $S = \{P(n): n \in \N\}$. The following are equivalent:
    \begin{enumerate}[label=(\roman*)]
        \item $S$ is an infinite subsemigroup of $(\Q^{>0}, \times)$.
        
        \item $S$ contains an infinite subsemigroup of $(\Q^{>0}, \times)$.
        
        \item $P(x) =  (ax+b)^d$ for some $a, d \in \N$, $b \in \Z$ with $a|b(b-1)$.
    \end{enumerate}
\end{theorem}
\begin{proof}
    (iii) $\Rightarrow$ (i): Assume $Q(x) = (ax + b)^d$ with $a | b(b-1)$. Then for any $x, y \in \N$, 
    \[
        Q(x) Q(y) = (ax + b)^d (ay + b)^d = \left( a \left( axy + bx + by + \frac{b(b-1)}{a} \right) + b \right)^d \in S.
    \]
    Hence $S$ is a multiplicative semigroup.
    
    (i) $\Rightarrow$ (ii): Trivial. 
    
    (ii) $\Rightarrow$ (iii): Assume that $S = \{P(n): n \in \N\}$ contains an infinite multiplicative semigroup of $\Q^{>0}$. 
    Then the leading coefficient of $P$ is positive.
    Write $P(x) = Q(ax + b)$ as in \cref{lem:factorP}, noting that $a>0$. 
    Then we have $\{P(n): n \in \N\} = \{Q(an+b): n \in \N\} \subset \{Q(n): n \in \N\}$. 
    It then follows that $\{Q(n): n \in \N\}$ contains a semigroup. 
    By \cref{prop:d-1}, $Q(x) = a_d x^d$ for some $d \in \N$ and $a_d \in \Q$. Since $\{Q(n): n \in \N\}$ contains a semigroup, there exist $x, y, z \in \N$ such that $Q(x) Q(y) = Q(z)$. 
    In other words, $a_d^2 (xy)^d = a_d z^d$, or $a_d = (z/(xy))^d$. 
    Let $c = z/(xy) \in \Q^{>0}$, we get $a_d = c^d$. 
    Hence $P(x) = (c(ax + b))^d = (ca x + cb)^d$.
    
    By abuse of notation, assume that $P(x) = (ax + b)^d$ for some $a, b \in \Q$. Since $P \in \Z[x]$, we have $a, b \in \Z$. Because $S$ contains an infinite multiplicative semigroup, there exists $D \in \N \setminus \{1\}$ such that for all $k \in \N$, there exists $x \in \N$ and $(ax + b)^d = D^{kd}$. In other words, $ax + b = D^k$. It follows that $D^k \equiv b \mod a$ for all $k \in \N$. Hence we have simultaneously that $D^{2k} \equiv b^2$ and $D^{2k} \equiv b \mod a$. This implies that $b^2 \equiv b \mod a$, or $a | b(b-1)$, which finishes the proof. 
\end{proof}




\subsection{Rational functions that do not contain a semigroup}

We prove \cref{prop:pythagorean_triples_semigroups} in this section. In addition, we show that many other rational functions of interest do not contain infinite multiplicative semigroups. 

\begin{lemma}
\label{lem:all_a_i_zero}
    Let $\ell\geq 2$, $q_1, q_2, \ldots, q_{\ell} \in \R\backslash\{0\}$ be non-zero and $|q_i| \neq |q_j|$ for all $i \neq j$. Let $a_1, \ldots, a_{\ell} \in \R$ be such that $a_1 q_1^k + \ldots + a_{\ell} q_{\ell}^k = 0$ for infinitely many $k \in \N$. Then $a_1 = a_2 = \ldots = a_{\ell} = 0$.
\end{lemma}
\begin{proof}
    By contradiction, assume some $a_i$ are non-zero. By discarding those $a_i$ that are zero, without loss of generality, we may assume that none of the $a_i$ are zero.
    
    Assume $|q_1| = \max \{|q_i|: 1 \leq i \leq \ell\}>0$. Then we have 
    \[
        |a_1 q_1^k| = |a_2 q_2^k + \ldots + a_{\ell} q_{\ell}^k|.    
    \]
    It follows that
    \[
        |a_1| = |a_2 (q_2/q_1)^k + \ldots + a_{\ell} (q_{\ell}/q_q)^k|
    \]
    for infinitely many $k \in \N$. But this is impossible since the right hand side approaches zero as $k \to \infty$.
\end{proof}

We are ready to prove \cref{prop:pythagorean_triples_semigroups}.

\begin{proof}[Proof of \cref{prop:pythagorean_triples_semigroups}]
    We only give a proof that the set $S = \{(m^2 + n^2)/(2mn): m, n \in \N\}$ does not contain an infinite subsemigroup of $(\Q^{>0}, \times)$ as the proof for the other set is the same. 
    By contradiction, assume that there exist $p, q \in \Z, p \neq q, (p,q) = 1$ such that $\{(p/q)^k: k \in \N\} \subset S$. By replacing $(p/q)$ by $(p/q)^2$, we can assume $p, q > 0$. Then it follows that for every $k \in \N$, there exist $m, n \in \N$ such that 
    \begin{equation}
    \label{eq:m_k}
        \frac{p^k}{q^k} = \frac{m^2 + n^2}{2 m n}.
    \end{equation}
    We can dividing both $m$ and $n$ by $gcd(m, n)$ and \eqref{eq:m_k} still holds. Therefore, without loss of generality, we may assume $(m, n) = 1$. Then
    \[
        2 p^k m n = q^k(m^2 + n^2).
    \]
    Since $(p,q) = 1$, it follows that $q^k | 2 m n$. On the other hand, since $(m, n) = 1$, $(m n, m^2 + n^2) = 1$. Thus $m n | q^k$. It follows that $q^k = m n$ or $q^k = 2 m n$. 
    
    If $q^k = m n$, then $2 p^k = m^2 + n^2$. Since $(m, n) = 1$, $m = q_1^k$ and $n = q_2^k$ for some $q_1 q_2 = q$. Hence
    \begin{equation}
    \label{eq:2p^k}
        2 p^k = (q_1^2)^k + (q_2^2)^k.
    \end{equation}
    
    If $q^k = 2 m n$, then $p^k = m^2 + n^2$. Similarly to above, $m = (q_1)^k/2$ and $n = q_2^k$ for some $q_1 q_2 = q$. We then have $p^k = q_1^{2k}/4 + q_2^{2k}$, or equivalently,
    \begin{equation}
    \label{eq:p^k}
        4 p^k = (q_1^2)^k + 4 (q_2^2)^k.
    \end{equation}
    
    Since the set $\{(q_1, q_2) \in \N^2: q_1 q_2 = q\}$ is finite, there must be some fixed $q_1, q_2$ such that either \eqref{eq:2p^k} or \eqref{eq:p^k} is true for infinitely many $k \in \N$. But, in view of \cref{lem:all_a_i_zero}, this is impossible, and hence we obtained the desired contradiction. 
\end{proof}

Let $\ell_1, \ell_2, \ell_3, \ell_4 \in \Z$ be pairwise distinct and let 
\begin{equation}
\label{eq:S}
    S = \left \{ \frac{(m + \ell_1 n)(m+ \ell_2 n)}{(m+ \ell_3 n)(m + \ell_4 n)}: m, n \in \N \right\}.
\end{equation}
It is shown in \cite{Frantzikinakis_Host_2017} that the set $S$ in \eqref{eq:S} is a set of measurable multiplicative recurrence. In the following proposition, we show that $S$ does not contain any infinite subsemigroup of $(\Q^{>0}, \times)$. 

\begin{proposition}
\label{prop:m_ell}
    Let $\ell_1, \ell_2, \ell_3, \ell_4 \in \Z$ be pairwise distinct. Then the set 
    \[
        S = \left \{ \frac{(m + \ell_1 n)(m+ \ell_2 n)}{(m+ \ell_3 n)(m + \ell_4 n)}: m, n \in \N \right\}
    \]
    does not contain any infinite subsemigroup of $(\Q^{>0}, \times)$. 
\end{proposition}
\begin{proof}
    By contradiction, assume there exist $p \in \Z, q \in \N$ with $(p, q) = 1$ such that for all $k \in \N$, there exist $m, n \in \N$ such that
    \[
        \frac{(m + \ell_1 n) (m + \ell_2 n)}{(m + \ell_3 n)(m + \ell_4 n)} = \frac{p^k}{q^k}.
    \]
    Dividing $m, n$ by $gcd(m,n)$, we can assume $(m,n) = 1$. We then have
    \[
        q^k (m + \ell_1 n)(m + \ell_2 n) = p^k (m + \ell_3 n)(m + \ell_4 n).
    \]
    It follows that $q^k | p^k (m+ \ell_3 n)(m + \ell_4 n)$. But since $(p,q) = 1$, it implies that $q^k | (m + \ell_3 n)(m + \ell_4 n)$. Similarly, $p^k | (m + \ell_1 n)(m + \ell_2 n)$. Let
    \[
        t = \frac{(m + \ell_1 n)(m + \ell_2 n)}{p^k} = \frac{(m + \ell_3 n)(m + \ell_4 n)}{q^k} \in \Z.
    \]
    Therefore $t | gcd((m + \ell_1 n)(m + \ell_2 n), (m + \ell_3 n)(m + \ell_4 n))$. 
    
    For $i,j \in \{1, 2, 3, 4\}$, let $d = gcd(m + \ell_i n, m + \ell_j n)$. Then $d | (\ell_j - \ell_i)n$ and $d | ((\ell_j - \ell_i)m + (\ell_j - \ell_i)\ell_i n)$. Therefore, $d | (\ell_j - \ell_i)m$. Since $(m,n) = 1$, it follows that $d | (\ell_j - \ell_i)$. We deduce that $t | (\ell_3 - \ell_1)(\ell_3 - \ell_2)(\ell_4 - \ell_1)(\ell_4 - \ell2)$. We then have
    \[
        (m + \ell_1n)(m + \ell_2 n) = t p^k
    \]
    and
    \[
        (m + \ell_3 n)(m + \ell_4 n) = t q^k.
    \]
    Since $(m + \ell_1 n, m + \ell_2 n) | (\ell_2 - \ell_1)$, we deduce that there exists some $k_0 = k_0(\ell_1, \ell_2, \ell_3, \ell_4)$ such that 
    \begin{equation}
    \label{eq:ell_12}
        m +\ell_1 n = u_1 p_1^{k-k_0} \mbox{ and } m + \ell_2 n = u_2 p_2^{k - k_0}
    \end{equation}
    where $(p_1, p_2) = 1$, $p_1 p_2 = p$ and $u_1 u_2 = t p^{k_0}$. Similarly, there exists $k_1 = k_1(\ell_1, \ell_2, \ell_3, \ell_4)$ such that
    \begin{equation}
    \label{eq:ell_34}
        m + \ell_3 n = v_1 q_1^{k - k_1}
        \mbox{ and }    
        m + \ell_2 n = v_2 q_2^{k - k_1}
    \end{equation}
    where $(q_1, q_2) = 1$, $q_1 q_2 = q$ and $v_1 v_2 = t q^{k_1}$. 
    
    Since there are only finitely many choices of $t, u_i,v_i, p_i, q_i$ for $i = 1, 2$, it follows that there exist fixed $u_i, v_i, p_i, q_i$ for $i = 1, 2$ for which there are infinitely many $k \in \N$ such that \eqref{eq:ell_12} and \eqref{eq:ell_34} simultaneously have solutions $(m, n) \in \N^2$. Solving these equations, we get
    \[
        n = \frac{u_2 p_2^{k-k_0} - u_1 p_1^{k- k_0}}{\ell_2 - \ell_1} = \frac{v_2 q_2^{k-k_1} - v_1 q_1^{k- k_1}}{\ell_3 - \ell_4}.
    \]
    In particular, 
    \[
        \frac{u_2}{(\ell_2 - \ell_1)p_2^{k_0}} \times p_2^k - \frac{u_1}{(\ell_2 - \ell_1) p_1^{k_0}} \times p_1^k - \frac{v_2}{(\ell_3 - \ell_4)q_2^{k_1}} \times q_2^k + \frac{v_1}{(\ell_3 - \ell_4) q_1^{k_1}} \times q_1^k = 0
    \]
    for infinitely many $k \in \N$. Since $p_1, p_2, q_1, q_2$ are pairwise coprime, we can apply \cref{lem:all_a_i_zero} to show that $u_i = v_i = 0$ for $i = 1, 2$. It follows from \eqref{eq:ell_12} that $m = n = 0$. This contradicts our assumption that $m, n \in \N$ finishing the proof. 
\end{proof}

\section{Additive averages of multiplicative recurrence sequences}
\label{sec:mix_ergodic_theory}

\subsection{Multiplicative subsemigroups of $\N$}

Recall from \cref{sec:sets_mult_rec} that a subset $S\subset G$ is syndetic if there exists a finite set $F\subset G$ such that $\bigcup_{x\in F}S/x= G$. 
As \cref{prop_syndetic} shows, every syndetic set $F\subset G$ has positive density with respect to any F\o lner sequence in $G$.
The next result shows that in the case when $(G,\times)$ is a subsemigroup of $(\N,\times)$, then a (multiplicative) syndetic subset of $G$ has positive ``upper density'' also with respect to the ``additive'' sequence $(G\cap[N])_{N \in \N}$.

\begin{lemma}
\label{lem:mult_syndetic_dense}
    Let $(G,\times)$ be a  subsemigroup of $(\N,\times)$ and $S$ be a syndetic subset of $G$. Then 
    \[
        \limsup_{N \to \infty} \E_{n \in G \cap [N]} 1_S(n) > 0.
    \]
    Moreover, if $G$ has positive lower additive density in $\N$, then
    \[
        \liminf_{N \to \infty} \E_{n \in G \cap [N]} 1_S(n) > 0.
    \]
\end{lemma}

\begin{proof}
     We first claim that for any non-empty set $E \subset \N$ and $c > 0$, we have
    \begin{equation}\label{123}
    \limsup_{N \to \infty} \frac{|E \cap [N]|}{|E \cap [cN]|} > 0.
    \end{equation}
    In fact, by contradiction, assume 
    	\[
    	\limsup_{N \to \infty} \frac{|E \cap [N]|}{|E \cap [cN]|} = 0.
    	\]
    	Then there exists $N_0$ such that for all $N \geq N_0$, we have
    	\[
    	\frac{|E \cap [N]|}{|E \cap [cN]|} < \frac{1}{2c}.
    	\]
    	Or equivalently,
    	\[
    	|E \cap [cN]| > 2c |E \cap [N]|.
    	\]
    	Inductively, we have for all $n \in \N$,
    	\begin{equation}
    	\label{eq:feb_20_3}
    	|E \cap [c^n N_0]| > (2c)^n |E \cap [N_0]|.
    	\end{equation}
    	But since $|E \cap [c^n N_0]| \leq c^n N_0$, \eqref{eq:feb_20_3} cannot hold for $n$ very large. This is a contradiction. This finishes the proof of the claim.
	
	We now return to the proof of \cref{lem:mult_syndetic_dense}.
  Since $S$ is syndetic in $G$, there exists a finite subset $F \subset G$ such that $S/F = G$. It follows that
    \[
        \sum_{m \in F} 1_S(mn) \geq 1 \,\,\, \mbox{ for all } n \in G.
    \]
    Therefore
    \begin{equation}
    \label{eq:feb-19-1}
        \sum_{n \in [N] \cap G} \sum_{m \in F} 1_S(mn) \geq |G \cap [N]|.
    \end{equation}
    Let $M = \max \{m: m \in F\}$. Since each $n \leq MN$ is repeated at most $|F|$ times in the sum appearing in the left hand side of \eqref{eq:feb-19-1}, we deduce that
    \[
        \sum_{n \in G \cap [MN]} 1_S(n) \geq \frac{|G \cap [N]|}{|F|}.
    \]
    Therefore
    \begin{equation}
    \label{eq:limsup_1}
        \E_{n \in G \cap [MN]} 1_S(n) \geq \frac{|G \cap [N]|}{|F| |G \cap [MN]|}
    \end{equation}
    In view of the claim,
    \begin{equation}
    \label{eq:feb_20_4}
        \limsup_{N \to \infty} \frac{|G \cap [N]|}{|G \cap [MN]|} > 0.
    \end{equation}
    \eqref{eq:feb_20_4} and \eqref{eq:limsup_1} imply $\overline{d}_G(S) > 0$.
    
    Assume $G$ has positive lower additive density. 
    In this case, we have
    \[
        \frac{|G \cap [N]|}{|G \cap [MN]|} = \frac{|G \cap [N]|}{N} \cdot \frac{N}{|G \cap [MN]|} \geq \frac{\delta}{M} \mbox{ for sufficiently large } N.
    \]
    Hence, for sufficiently large $N$,
    \begin{equation}
    \label{eq:ave_MN}
        \E_{n \in G \cap [MN]} 1_S(n) \geq \frac{\delta}{M |F|}.
    \end{equation}
    We now have
    \begin{multline}
    \label{eq:est}
        \E_{n \in G \cap [N]} 1_S(n) = \frac{1}{|G \cap [N]|} \sum_{n \in G \cap [N]} 1_S(n) \geq \\\frac{|G \cap [M \lfloor N/M \rfloor|}{|G \cap [N]|} \cdot \frac{1}{|G \cap [M \lfloor N/M \rfloor]|} \sum_{n \in G \cap [M \lfloor N/M \rfloor]} 1_S(n) 
        = \frac{|G \cap [M \lfloor N/M \rfloor|}{|G \cap [N]|} \E_{n \in G [M \lfloor N/M \rfloor]} 1_S(n).
    \end{multline}
    Moreover,
    \begin{equation}
    \label{eq:frac_MN}
        \frac{|G \cap [M \lfloor N/M \rfloor|}{|G \cap [N]|} \geq \frac{|G \cap [M \lfloor N/M \rfloor|}{M \lfloor N/M \rfloor} \cdot \frac{M \lfloor N/M \rfloor}{N} > \frac{|G \cap [M \lfloor N/M \rfloor|}{M \lfloor N/M \rfloor} \cdot \frac{M (N/M - 1)}{N} > \frac{\delta}{2} 
    \end{equation}
    for sufficiently large $N$. Combining \eqref{eq:ave_MN}, \eqref{eq:est} and \eqref{eq:frac_MN}, for sufficiently large $N$,
    \[
        \E_{n \in G \cap [N]} 1_S(n) > \frac{\delta^2}{2 M |F|}
    \]
    This implies $\underline{d}_G(S) > 0$. This finishes our proof.
\end{proof}

\begin{remark}
    \cref{lem:mult_syndetic_dense} implies that a multiplicatively syndetic set in $\N$ has positive upper  density. On the contrary, it is not true that an additively syndetic set has positive upper multiplicative density. 
    For example, the set $2 \N + 1$ is additively syndetic but has zero multiplicative density with respect to any multiplicative F{\o}lner sequence. 
\end{remark}

	It is worth noting that in (\ref{123}), one can not replace $\limsup$ with $\liminf$. In fact, we have:

\begin{proposition}
	There exists a subsemigroup $G$ of $(\N, \times)$ such that
	\[
	\liminf_{N \to \infty} \frac{|G \cap [N]|}{|G \cap [2N]|} = 0.
	\]
\end{proposition}
\begin{proof}
	We construct $G$ by putting primes into $G$ in an inductive way as follows. 
	Let $1 \leq N_1 < N_2 < \ldots$ be a sequence of integers with the exact values of $N_k$ to be chosen later.
	Put all primes in the intervals $[N_1, 2 N_1] \cup [N_2, 2 N_2] \cup \ldots$ into $G$. Other elements of $G$ are generated from those primes. Then for $k \in \N$, the set $G \cap [N_{k+1}]$ is subset of the set $\{p_1^{\ell_1} \cdots p_m^{\ell_m}: p_1, \ldots, p_m \in \P \cap [N_k], m \in \N\}$ where $\P$ denotes the set of primes. Since $|\P \cap N_k| < N_k$, we can see
	\[
	    |G \cap [N_{k+1}]| < (\log_2 N_{k+1})^{N_k}<(2\log N_{k+1})^{N_k}.
	\]
	On the other hand, using the Prime Number Theorem, we have
	\[
	    \big|G \cap [N_{k+1}, 2 N_{k+1}]\big| \geq\big|\P\cap[N_{k+1}, 2 N_{k+1}]\big|
	    = \frac{N_{k+1}}{\log N_{k+1}}\big(1+o_{k\to\infty}(1)\big)
	\]
	Hence
	\begin{equation}
	\label{eq:feb_20_1}
	\frac{|G \cap [N_{k+1}]|}{|G \cap [N_{k+1}, 2 N_{k+1}]|}\leq \frac{(2\log N_{k+1})^{N_k + 1}}{N_{k+1}}\big(1+o_{k\to\infty}(1)\big)
	\end{equation}
	If we choose the sequence $(N_k)$ growing fast enough so that the quantity on the right hand side of \eqref{eq:feb_20_1} approaches zero as $k \to \infty$, then we have the conclusion.  
\end{proof}

\subsection{Averages along subsemigroups of $(\N,\times)$}

\cref{prop:positive_semigroup_intro} now follows from the next result.

\begin{theorem}
	\label{prop:positive_semigroup}
	Let $(X, \mathcal{B}, \mu, T)$ be a measure preserving $(\N, \times)$-system and $A \in \mathcal{B}$ with $\mu(A)>0$. Let $(G,\times)$ be a subsemigroup of $(\N, \times)$.
	Then for every $\ell \in \N$,
	\[
	\limsup_{N\to\infty} \E_{n \in G \cap [N]} \mu(A \cap T_{n}^{-1} A \cap \ldots \cap T_{n^{\ell}}^{-1} A)>0.
	\]
	If $\liminf_{N\to\infty}\frac{\vert G\cap[N]\vert}{N}> 0$ (i.e. $G$ has positive lower density), then we have the stronger bound
	\[
	\liminf_{N\to\infty} \E_{n \in G \cap [N]} \mu(A \cap T_{n}^{-1} A \cap \ldots \cap T_{n^{\ell}}^{-1} A)>0.
	\]
\end{theorem}   
\begin{remark}
    \cref{prop:positive_semigroup} is no longer true if we replace the F{\o}lner sequence $([N])_{N \in \N}$ by arbitrary additive F{\o}lner sequence on $\N$. 
    For example, take $(X, \mathcal{B}, \mu, T)$ to be the $(\N, \times)$-system where $X = \T$, $\mu$ is the Lebesgue measure and $T_n x \coloneqq x + \log n$, and let $A$ be a small interval on $X$. Then there exist arbitrary long intervals $[M, N)$ of integers such that $A \cap T^{-1}_n A = \emptyset$ for all $n \in [M, N)$.
    Therefore we can find a F\o lner sequence $\Phi$ consisting of longer and longer such intervals with the property that 
    $$\lim_{N\to\infty}\E_{n\in\Phi_N}\mu(A\cap T_nA)=0.$$
\end{remark}

\begin{proof}
     In view of \cref{thm_Furstenberg_katznelson}, there exists $\delta > 0$ such that the set
    \[
        S = \{n \in G: \mu(A \cap T_{n}^{-1} A \cap \ldots \cap T_{n^{\ell}}^{-1} A) > \delta\}
    \]
    is syndetic in $(G, \times)$. 
    We then have
    \[
        \E_{n \in G \cap [N]} \mu(A \cap T_{n}^{-1} A \cap \ldots \cap T_{n^{\ell}}^{-1} A) \geq \delta \E_{n \in G \cap [N]} 1_S.
    \]
    \cref{prop:positive_semigroup} now follows from \cref{lem:mult_syndetic_dense}.
\end{proof}

We conclude this section by
 remarking that one can not replace the $\liminf$ and $\limsup$ in Theorem \ref{prop:positive_semigroup} with limits. This is because the limits may not exist as the following proposition shows.

\begin{proposition}
	\label{prop:not_exist}
	There exists a measure preserving system $(\N, \times)$-system $(X, \mathcal{B}, \mu, T)$ and a set $A \in \mathcal{B}$ such that 
	\begin{equation}
	\label{eq:not_exist}
	\lim_{N \to \infty} \E_{n \in [N]} \mu(A \cap T_n^{-1} A)
	\end{equation}
	does not exist.
\end{proposition}
\begin{proof}
	Let $\chi\colon \N \to S^1$ be the completely multiplicative function $\chi(n) = n^{i} = e^{i \log n}$ for $n \in \N$. It is well known that 
	\begin{equation}
	\label{eq:lim_chi}
	\lim_{N \to \infty} \E_{n \in [N]} \chi(n)
	\end{equation}
	does not exist. 
	Consider the $(\N,\times)$-system $(X = S^1, \mu, (T_n)_{n\in\N})$ where $\mu$ is the normalized Lebesgue measure on $X$ and $T_n(x) := \chi(n) x $ for $x \in X$ and $n \in \N$. 
	We show that this system satisfies our condition. By contradiction, assume for every $A \subset X$ measurable, the limit in \eqref{eq:not_exist} exists. 
	
	Let $A$ be an arbitrary measurable subset of $X$. Let $\mathcal{H}$ be the smallest closed subspace of $L^2(X)$ that contains $T_n 1_A$ for all $n \in \N$ and constant functions. By our assumption, for every $g \in \mathcal{H}$,
	\[
	\lim_{N \to \infty} \E_{n \in [N]} \int_X 1_A(T_n x) \overline{g}(x) \, d \mu(x)
	\]
	exists. Let
	\[
	\mathcal{F} = \left \{ f \in L^2(X): \lim_{N \to \infty} \E_{n \in [N]} \int_X 1_A(T_n x) \cdot \overline{f}(x) \, d \mu(x) \mbox{ exists }\right\}.
	\]
	It follows that $\mathcal{H} \subset \mathcal{F}$.
	On the other hand, for all $g \in \mathcal{H}^{\perp}$, we have
	\[
	\int_X 1_A(T_n x) \cdot \overline{g}(x) \, \mu(x) = 0.
	\]
	Hence 
	\[
	\lim_{N \to \infty} \E_{n \in [N]} \int_X 1_A(T_n x) \cdot \overline{g}(x) \, d \mu(x) 
	\]
	exists and is equal to $0$. In particular, $\mathcal{H}^{\perp} \subset \mathcal{F}$. Therefore $\mathcal{F} = L^2(X)$. In conclusion, for every $A \in\mathcal{B}$  and every $f \in L^2(X)$, the limit
	\[
	\lim_{N \to \infty} \E_{n \in [N]} \int_X 1_A(T_n x) \cdot \overline{f}(x) \, \mu(x) 
	\]
	exists. By approximating $f$ with simple functions, we get for all $f \in L^2(X)$, the limit
	\[
	\lim_{N \to \infty} \E_{n \in [N]} \int_X f(T_n x) \cdot \overline{f}(x) \, \mu(x)
	\]
	exists. However, this is not true as the following example shows. Let $f\colon X \to \C$ be defined by $f(x) = x$ for $x \in S^1$. Then $f(T_n x) = \chi(n) x$. We then have
	\[
	\int_{X} f(T_n x) \overline{f}(x)  d \mu = \chi(n).
	\]
	Hence the limit
	\begin{equation}
	\label{eq:w_lim_f}
	\lim_{N \to \infty} \E_{n \in [N]} \int_{X} f \cdot T_n \overline{f} \, d\mu
	\end{equation}
	does not exist. We reach a desired contradiction.
\end{proof}
\begin{remark}
One may ask whether the limit in \eqref{eq:not_exist} might exist after replacing Ces\`aro averages with a different averaging scheme (for example logarithmic averages).
However, for any given averaging scheme, there exists a complete multiplicative function $\chi$ for which the limit in  \eqref{eq:lim_chi} does not exist. Hence we can carry a similar construction as in proof of \cref{prop:not_exist} to show that there exists a system and set $A$ for which the limit in \eqref{eq:not_exist} with respect to such an averaging scheme does not exist.
\end{remark}

\subsection{Subordinated semigroups}

In this section, we introduce the notion of subordinated semigroups of $(\N^{k},\ast)$ and some basic properties of these semigroups for later uses.

\begin{definition}[Subordinated semigroups] We say that a semigroup $(\N^{k},\ast)$ is {\em subordinated} to the Euclidean norm $\|\cdot \|$ in $\N^k$ if there exists a constant $C>0$ such that 
\begin{equation}\label{dss}
\| n*m \| \leq C \|n\|\|m\| \text{ for all } n,m\in \N^k.
\end{equation}	
\end{definition}

We remark that (\ref{dss}) is equivalent of saying that $\sup_{n,m \in \N^k} \frac{\|n*m\|}{\|n\|\|m\|}$ is finite. Also note that, up to modifying the constant $C$, in the definition of subordinated semigroup we may change the Euclidean norm $\Vert\cdot\Vert$ by any other equivalent norm.

The following lemma follows immediately from the definition. 
\begin{lemma}\label{lemma:Lipschitz}
	Let $(\N^k, \ast)$ be a semigroup subordinated to the Euclidean norm in $\N^k$. Then for every $m\in \N^k$, there exists $K=K(m)\in\N$ such that for all $N\in\N$ and $n\in \N^{k}$ with $\Vert n\Vert\leq N$, we have that $\|m*n\|\leq KN$.
\end{lemma} 

A class of subordinated semigroups of particular interest for us is given by semigroups $(\N^k,\ast)$ whose operation is induced by the usual matrix multiplication in a subsemigroup of $\mathcal{M}_{d\times d}(\Z)$ for some $d\in \N$. Here $\mathcal{M}_{d\times d}(\Z)$ denotes the set of $d \times d$-matrices with integer coefficients.

\begin{definition}\label{def_determinantefunction}
Let $d,k\in\N$. We say that a semigroup $(\N^k,\ast)$ is induced by $\mathcal{M}_{d\times d}(\Z)$ (via $\psi$) if there exists a linear injection $\psi \colon \N^k\to \mathcal{M}_{d\times d}(\Z)$ (meaning that $\psi$ is a linear function of $\N^{k}$ in every entry) such that $\psi(\N^k)$ is a semigroup of $\mathcal{M}_{d\times d}(\Z)$ not contained in $\{A\in \mathcal{M}_{d\times d}(\Z) : \det(A)=0 \}$ and such that 
\[n\ast m=\psi^{-1}(\psi(n) \psi(m)).\]
\end{definition}

\begin{lemma}\label{idc}
A semigroup $(\N^k,\ast)$ induced by $\mathcal{M}_{d\times d}(\Z)$ is subbordinated to the Euclidean norm in $\N^{k}$. 
\end{lemma}
\begin{proof} We regard $\mathcal{M}_{d\times d}(\Z)$ as a subset of $\R^{d^2}$, and let $N\in M_{k,d^2}(\R)$ and  $M \in M_{d^2,k}(\R)$ be such that $\psi(x)=Nx$ and $\psi^{-1}(y)=My$ for $x\in \N^k$, $y\in \psi(\N^k)$. For any matrix $A\in M_{d_1,d_2}(\R)$, denote $\|A\|=\sup_{x\in \R^{d_1}\setminus\{ 0\}} \|Ax\|/\|x\|$. We have that \[
     \|n \ast m\| =\|\psi^{-1}(\psi(n)\cdot \psi(m)) \| 
     \leq  \|M\| \|\psi(n)\cdot \psi(m)\| 
     \leq  \|M\| \|\psi(n)\|\|\psi(m)\| 
      \leq \|M\| \|N\|^2 \|n\| \|m\|.\]
\end{proof}

The next lemma states that a (multiplicative) syndetic subset in a semigroup subordinated to the Euclidean norm has positive (additive) lower density; this phenomenon was already observed in \cref{lem:mult_syndetic_dense}.

\begin{lemma} \label{lemma_syndetic_liminf}
	Let $(\N^k,\ast)$ be a semigroup subordinated to the Euclidean norm in $\N^k$ and $(Z,\ast)$ be a subsemigroup of $(\N^k,\ast)$ with 
	$$\overline{d}_{+}(Z) \coloneqq \limsup_{N\to\infty}\frac{\big|[N]^k\cap Z\big|}{N^k}=1.$$
	Let $S\subseteq Z$ be a syndetic subset (with respect to the $\ast$ operation). Then,
	\begin{equation*}
	\liminf_{N\to\infty}\E_{n\in[N]^k\cap Z}1_S(n)>0.
	\end{equation*}
\end{lemma}

\begin{proof}
	Since $S$ is syndetic, there exists a finite set $F\subset Z$ such that $Z=F^{-1}\ast S$. 	Then for every $x\in Z$ we have that $1\leq\sum_{n\in F}1_S(n\ast x)$.
	By \cref{lemma:Lipschitz}, there exist $K=\max\{K(n)\colon n\in F\}$ and $N_{0}:=\max\{N_{0}(n)\colon n\in F\}$ such that for all $N\geq N_{0}$, $m*n\in f([KN]^k)$ for all $n\in [N]^k$ and $m\in F$.
	So for all $N\geq N_{0}$,
	$$1\leq\liminf_{N\to\infty}\E_{n\in[N]^k\cap Z}\sum_{m\in F}1_S(m*n)=\liminf_{N\to\infty}\sum_{m\in F}\E_{n\in[N]^k\cap Z}1_S(m*n)\leq\liminf_{N\to\infty}\frac1{N^k}\sum_{n\in[KN]^k\cap Z}1_S(n),$$
	where in the last inequality we used the fact that the map $n\to m*n$ is injective from $Z$ to $Z$ for all $m\in Z$. 
	Therefore,
	\begin{equation}\label{eq_wenbo123}
	\liminf_{N\to\infty}\E_{n\in[mN]^k\cap Z}1_S(n)>K^{-k}>0.
	\end{equation}
	Now for all $N\geq KN_{0}$, let $M=[N/K]$, then (\ref{eq_wenbo123}) implies that
	\begin{multline*}
	\liminf_{N\to\infty}\E_{n\in[N]^k\cap Z}1_S(n)\geq \liminf_{N\to\infty}\Bigl(\frac{KM}{N}\Bigr)^{k}\E_{n\in[KM]^k\cap Z}1_S( n)\\
	>\liminf_{N\to\infty}\Bigl(\frac{KM}{N}\Bigr)^{k}K^{-k}>\liminf_{N\to\infty}\Bigl(\frac{1}{K}-\frac{1}{N}\Bigr)^{k}.
	\end{multline*}
	So $\liminf_{N\to\infty}\E_{n\in[N]^k\cap Z}1_S(n)\geq K^{-k}$ and we are done.
\end{proof}



\subsection{Parametrized multiplicative functions}
\label{sec_ParametrizedMultiplicativeFunctions}

In this section, we introduce a class of functions that parametrize multiplicative subsemigroups of $\Q^{>0}$.
\begin{definition}[Parametrized multiplicative function]\label{def_pmf}
For $k\in\N$, we say that $f\colon \N^k\to \Q^{\geq0}$ is a \emph{parametrized multiplicative function} if 
\[
    \bar d_+\big(\{n\in\N^k:f(n)=0\}\big)=0
\]
and there exists an operation $\ast\colon \N^k\times \N^k\to \N^k$ so that $(\N^{k},\ast)$ is a semigroup subordinated to the euclidean norm in $\N^k$ and
\[f(n*m)=f(n) f(m) \]
for all $n,m\in \N^k$. 
We say that $f$ is a \emph{commutative parametrized multiplicative function} if one can choose $(\N^{k},\ast)$ to be a cancelative commutative semigroup.
\end{definition}

\begin{remark}\label{remark_averagedensity}
    We could disallow functions $f:\N^k\to\Q^{\geq0}$ that take the value $0$, but this would forbid some of the natural examples we exhibit below. 
    On the other hand, we are interested in functions that take values in the multiplicative group $(\Q^{>0},\times)$, so as a compromise we allow parametrized multiplicative functions to take the value $0$, as long as only on a set of zero additive density.
    
    This definition leads to a situation where we may have a function $f:\N^k\to\R$ which is only defined on a full density subset $D$ of $\N^k$; in such cases we denote by
    $$\liminf_{N\to\infty}\E_{n\in[N]^k} f(n) := \liminf_{N\to\infty}\E_{n\in[N]^k}1_D(n) f(n).$$
\end{remark}

We sometimes write $f\colon (\N^k,\ast)\to \Q^{\geq0}$ for a parametrized multiplicative function to stress the semigroup operation for which $f$ is a multiplicative function.

A typical example of a parametrized multiplicative function arising from a semigroup $(\N^k,\ast)$ induced by $\mathcal{M}_{d\times d}(\Z)$ is the one induced by the determinant function. 
That is $f\colon (\N^{k},\ast)\to \Q^{\geq 0}$, $n\mapsto \vert\det(\psi(n))\vert$, or more generally $n\mapsto \xi(\vert\det(\psi(n))\vert)$ for any multiplicative function $\xi\colon \vert\det(\psi(\N^k))\vert\to \Q^{\geq0}$.
Here we point out that since $\psi(\N^k)$ is not contained in $\{A\in \mathcal{M}_{d\times d}(\Z) : \det(A)=0 \}$, the polynomial $\det(\psi(n))$ is non trivial and therefore $\bar d\big(\{n\in \N^{k}: \det(\psi(n))=0\}\big)=0$. 
(To see this, note that for $(n_1,\ldots,n_{k-1})\in \N^{k-1}$ there are at most $d$ positive integers $n_k\in \N$ such that $\det(\psi(n_1,\ldots,n_{k-1},n_k))=0$.)

\begin{example}
In the following, we present some examples of parametrized multiplicative functions.
\begin{enumerate}
	\item Consider the function $f$ given by $f=\det\circ \psi$, where $\psi\colon \N^{k}\to \mathcal{M}_{2\times 2}(\Z)$, $(n_1,n_2)\mapsto \begin{pmatrix}
	n_1 & -n_2 \\ n_2 & n_1
	\end{pmatrix}$ is a linear injection.
	By Lemma \ref{idc}, the semigroup induced by $\mathcal{M}_{2\times 2}(\Z)$ via $\psi$ is subbordinated to the Euclidean norm in $\N^{2}$. Thus
	we have that $f(n_1,n_2)=n_1^2+n_2^2$ is a parametrized multiplicative function.

	\item If $k=a^2+b^2$ for some $a,b\in \N$, then $f(n_1,n_2)=k(n_1^2+n_2^2)$ is a parametrized multiplicative function, since $f=\det\circ \psi$ with $\psi\colon\N^{2}\to\mathcal{M}_{2\times 2}(\Z)$ being the linear injection $(n_1,n_2)\mapsto \begin{pmatrix}
	a & -b \\ b & a
	\end{pmatrix} \begin{pmatrix}
	n_1 & -n_2 \\ n_2 & n_1
	\end{pmatrix}$.
	
	\item Similar to the previous examples, the linear injection $\psi: (n_1,n_2)\mapsto \begin{pmatrix}
	n_1 & -Dn_2 \\ n_2 & n_1
	\end{pmatrix}$ for $D \in \Z$ gives us the paramatrized multiplicative function  $f(n_1,n_2) = \vert \det \circ \psi(n_1,n_2) \vert = \vert n_1^2+Dn_2^2 \vert$. 
	
	\item \label{remark:product_of_PMF}
	Given two parametrized multiplicative functions, $f_1\colon (\N^{d_1},\ast_1)\to \Q^{\geq0}$ and $f_2\colon (\N^{d_2},\ast_2)\to \Q^{\geq0}$, we may define for $\ell_1,\ell_2\in \Z$  the function $f\colon(\N^{d_1}\times \N^{d_2},\ast_1\times \ast_2)\to \Q$, such that $f(n_1,n_2)=f_1^{\ell_1}(n_1)f_2^{\ell_2}(n_2)$, if $f_1(n_1)f_2(n_2)\neq 0$ and $f(n_1,n_2)=0$ otherwise.
	Since $(\N^{d_i},\ast_i)$, $i=1,2$ is subordinated to the Euclidean norm in $\N^{d_i}$, we get that $(\N^{d_1}\times \N^{d_2}, \ast_1\times \ast_2)$ is subordinated to the Euclidean norm in $\N^{d_1}\times \N^{d_2}$.
	It is not hard to check that the upper density of $\{(n_1,n_2)\in \N^{d_1}\times \N^{d_2}: f(n_1,n_2)=0\}$ equals 0.
	This implies that $f$ is a parametrized multiplicative function. This also shows that $1/f_1$ is a parametrized multiplicative function.
	
	In particular, for any $p,q\in\Z$ the function $f\colon \N^2\to \Q^{>0}$, $(m,n)\mapsto m^{p} n^q$ is a parametrized multiplicative function.

	\item Let $K$ be a finite field extension of $\mathbb{Q}$ and $\mathcal{O}_{K}$ be the ring of integers of $K$. 
	Let $d=[K: \mathbb{Q}]$ and $\mathcal{B}=\{b_{1},\dots,b_{d}\}$ be an integral basis of $\mathcal{O}_{K}$, i.e., every $x\in\mathcal{O}_{K}$ can be written as $x=n_{1}b_{1}+\dots+n_{d}b_{d}$ for some $n_{1},\dots,n_{d}\in\mathbb{Z}$ in a unique way.
	Let $N_{K}(x)$ denote the norm of $x$ in $K$. 
	Since $N_{K}(x)N_{K}(y)=N_{K}(xy)$, the function $f\colon\N^{d}\to\Q^{>0}$ given by $f(n_{1},\dots,n_{d})=\big|N_{K}(n_{1}b_{1}+\dots+n_{d}b_{d})\big|$ is a parametrized multiplicative function.
			We remark that in this example, $f$ can be written as $f=\det\circ \psi$ for some linear injection $\psi\colon\mathbb{Z}^{d}\to\mathcal{M}_{d\times d}(\Q)$. 

    As a special case, let $\omega=e^{2\pi i /3}$ and $K=\mathbb{Q}(\omega)$. Then $\{1,\omega,\omega^{2}\}$ is an integral basis of $\mathcal{O}_{K}$. For every $x=a_{0}+a_{1}\omega+a_{2}\omega^{2}$ and $y=b_{0}+b_{1}\omega+b_{2}\omega^{2}$ with $a_{0},a_{1},a_{2},b_{0},b_{1},b_{2}\in\Q$, note that 
    \begin{equation}\nonumber
    \begin{split}
    &\quad xy=(a_{0}b_{0}+a_{1}b_{2}+a_{2}b_{1})+(a_{0}b_{1}+a_{1}b_{0}+a_{2}b_{2})\omega+(a_{1}b_{1}+a_{0}b_{2}+a_{2}b_{0})\omega^{2}
    \\&=
    \begin{pmatrix}
    b_{0} & b_{1} & b_{2}
    \end{pmatrix}
    \cdot \begin{pmatrix}
    a_{0} & a_{1} & a_{2}  \\
    a_{2} & a_{0} & a_{1}  \\
    a_{1} & a_{2} & a_{0}  
    \end{pmatrix}
    \cdot
    \begin{pmatrix}
    1 \\
    \omega \\
    \omega^{2} 
    \end{pmatrix}.
    \end{split}
    \end{equation}
    Therefore the function
     \begin{equation}\nonumber
     \begin{split}
     &\quad f(a_{0}, a_{1}, a_{2}):= \vert N_{K}(a_{0}+a_{1}\omega+a_{2}\omega^{2}) \vert = \left \vert \det \begin{pmatrix}
     a_{0} & a_{1} & a_{2}  \\
     a_{2} & a_{0} & a_{1}  \\
     a_{1} & a_{2} & a_{0}  
     \end{pmatrix} \right \vert
     \\&= \vert (a_{0}+a_{1}+a_{2})(a_{0}^{2}+a_{1}^{2}+a_{2}^{2}-a_{0}a_{1}-a_{0}a_{2}-a_{1}a_{2}) \vert = \vert a_{0}^{3}+a_{1}^{3}+a_{2}^{3}-3a_{0}a_{1}a_{2} \vert
     \end{split}
     \end{equation}
     is a parametrized multiplicative function.
	
	\item Let $p\colon\Z\to\Z$ be the minimal polynomial of $A\in \mathcal{M}_{n\times n}(\Z)$ and $d=\operatorname{deg}(p)$. Let $\psi\colon \N^{d}\to \mathcal{M}_{n\times n}(\Z)$ be such that $\psi(n_{0},\dots,n_{d-1}):=n_{0}I_{n}+n_{1}A+\dots+n_{d-1}A^{d-1}$ (which is a linear injection since $d$ is the degree of the minimal polynomial of $A$).
	Since $p(A)=0$,  $\psi(\N^d)$ is a semigroup and $\psi^{-1}\colon \psi(\N^d)\to \N$ is a well defined map. So we may
	denote $$n\ast m:=\psi^{-1}(\psi(n)\psi(m)).$$
	  Also note that $\psi(\N^d)$ does not consists of only non-invertible matrices.
	Then the group $(\N^{d},\ast)$ is induced by $\mathcal{M}_{n\times n}(\Z)$ (via the map $\psi$), and so $\left| \det\circ \psi \right| \colon \N^{d}\to \Z$ is a parametrized multiplicative function.
	
	To illustrate this class of examples, consider the Fibonacci matrix $A=\begin{pmatrix} 1 & 1 \\ 1 & 0 \end{pmatrix}$. Its minimal polynomial is of degree 2 and then we can consider $\psi\colon \N^2\to \mathcal{M}_{2\times 2}(\Z)$, given by
	\[\psi(n,m)=nI + m A= \begin{pmatrix} n+m & m \\ m & n \end{pmatrix}  \]
	Hence $|\det \circ \psi(n,m)|= |n^2 - m^2 + nm|$ is a parametrized multiplicative function.
	More generally, consider a matrix $A=\begin{pmatrix} a & b \\ c & d \end{pmatrix}$ not being a multiple of the identity. Then its minimal polynomial is of degree 2 and therefore 
	\[\left|\det(n I + m A)\right| =\left|\det \begin{pmatrix} n+am & bm \\ cm & n+dm \end{pmatrix}   \right|=| n^2 +(a+d)mn + (ad-bc)m^2 |\] is a parametrized multiplicative function.
	
		\item  Consider the quaternion semigroup of matrices $\{M_{n_{1},n_{2},n_{3},n_{4}}\colon n_{1},n_{2},n_{3},n_{4}\in \N\}$ where 
		\[M_{n_{1},n_{2},n_{3},n_{4}}=\begin{pmatrix}
		n_{1} & -n_{2} & -n_{3} & -n_{4} \\
		n_{2} & n_{1} & -n_{4} & n_{3} \\
		n_{3} & n_{4} & n_{1} & -n_{2} \\
		n_{4} & -n_{3} & n_{2} & n_{1}
		\end{pmatrix}.\]
		Via the linear injection $\psi\colon \N^4\to \mathcal{M}_{4\times 4}(\Z)$ with $(n_1,n_2,n_3,n_4)\mapsto M_{n_1,n_2,n_3,n_4}$ we obtain the parametrized multiplicative functions 
		$f(n_1,n_2,n_3,n_4)=\det \circ\psi(n_1,n_2,n_3,n_4)= (n_1^2+n_2^2+n_3^2+n_4^2)^2$ and $g(n_1,n_2,n_3,n_4)=(\det)^{1/2} (\psi(n_1,n_2,n_3,n_4))= n_1^2+n_2^2+n_3^2+n_4^2$.

\end{enumerate}
We remark that the functions in examples (1)--(6) are all \emph{commutative} parametrized multiplicative functions, however it is not clear whether the one in example (7) is, as the quaternion semigroup is non-commutative. 
\end{example}


\subsection{Averages along parametrized multiplicative functions}
In this section we prove
Theorem \ref{thm:multipledeterminant_intro}. 
In fact we will establish the following more general version which also deals with commutative parametrized functions that take values in $\Q^{\geq0}$.
Recall the convention adopted in \cref{remark_averagedensity}.

\begin{theorem}
\label{thm:multipledeterminant} 
	Let $(X, \mathcal{B}, \mu,T)$ be a measure preserving $(\N, \times)$-system and $A \in \mathcal{B}$ with $\mu(A)>0$. 
	Let $f\colon \N^k\to\Q^{\geq0}$ be a commutative parametrized multiplicative function and $\ell\in\N$.
	Then
	\[
	    \liminf_{N\to\infty}\E_{n\in[N]^k} \mu\Big(A\cap T_{f(n)}^{-1} A \cap T_{f(n)^2}^{-1} A \cap \cdots \cap  T_{f(n)^{\ell}}^{-1} A\Big)>0.
	 \]
\end{theorem}
\begin{proof}
The proof resembles that of \cref{prop:positive_semigroup}.  Let $(\N^k,\ast)$ be the semigroup that gives rise to the parametrized multiplicative function $f$. 
Let $Z$ be the set of all $n\in\N^{k}$ with $f(n)>0$. 
Since $f$ is a homomorphism, $(Z,\ast)$ is a subsemigroup of $(\N^{k},\ast)$.
For each $n\in Z$, let $R_{n}:=T_{f(n)}$.
Denote by $n^{\ast k}$ the iterate product $n^{\ast k}=n\ast n\cdots \ast n$ ($k$-times). 
By \cref{thm_Furstenberg_katznelson},
there exists $\delta > 0$ such that the set
    \[
        S = \{n \in Z: \mu(A \cap R_{n}^{-1} A \cap \ldots \cap R_{n^{\ast \ell}}^{-1} A) > \delta\}
    \]
    is syndetic in $(Z, \ast)$. Since $\overline{d}_{+}(Z)=1$ by definition, we have 
    \begin{align*}
    &\quad
    \liminf_{N\to\infty}\E_{n \in [N]^k}\mu(A\cap T_{f(n)}^{-1} A\cap T_{f(n)^2}^{-1} A \cap \cdots T_{f(n)^\ell}^{-1} A)
    \\&=
    \liminf_{N\to\infty}\E_{n\in [N]^{k}\cap Z}\mu(A\cap T_{f(n)}^{-1} A\cap T_{f(n)^2}^{-1} A \cap \cdots T_{f(n)^\ell}^{-1} A)
        \\&\geq \delta \liminf_{N\to\infty}\E_{n\in [N]^{k}\cap Z} 1_S(n) > 0,
    \end{align*}
where the last inequality comes from \cref{lemma_syndetic_liminf}.    
\end{proof}

%
%

As an immediate corollary of Theorem \ref{thm:multipledeterminant_intro}, we conclude that for all  multiplicative  measure preserving system $(X, \mathcal{B}, \mu,T)$ and $A \in \mathcal{B}$ with $\mu(A)>0$,
the following averages are positive:
\begin{enumerate}[ref=\arabic*]
    \item $\liminf_{N \to \infty} \E_{m,n \in [N]} \mu(T_{m^a}^{-1} A \cap T_{n^b}^{-1} A)$ for $a, b \in \N$;\footnote{We do not know how to establish this without using \cref{thm:multipledeterminant_intro} (i.e. using only \cref{prop:positive_semigroup_intro}).}
	\item $\liminf_{N\to\infty}\E_{n,m\in[N]}\mu(A\cap T_{n^{2}+m^{2}}^{-1} A)$ (which implies Proposition \ref{prop:multiple_squares_intro}); 
	
	\item $\liminf_{N\to\infty}\E_{n,m\in[N]}\mu(A\cap T_{n^{2}+Dm^{2}}^{-1} A)$ for all $D\in\N$;
	
	\item \label{item4} $\liminf_{N\to\infty}\E_{n,m\in[N]}\mu(A \cap T_{k(n^{2}+m^{2})}^{-1} A \cap T_{(k(n^{2}+m^{2}))^2} A )$ for all $k$ of the form $a^{2}+b^{2}$ for some $a, b \in \Z$.
\end{enumerate}	 
\begin{remark}
	We remark that item \eqref{item4} above fails if $k$ does not have the form $a^2 + b^2$. Indeed, if this is the case, there exists a prime $p \equiv 3 \pmod 4$ such that $p$ has odd exponent in prime factorization of $k$. 
	For $n \in \N$ let $f_p(n)$ denote the exponent of $p$ in the prime factorization of $n$. We then have $f_p(k(m^2 + n^2))$ is odd for all $m, n \in \N$. 
	For $n \in \N$, color $n$ with $C(n) \equiv f_p(n) \pmod 2$. It follows that $\{k (m^2 + n^2): m, n \in \N\}$ is  not a set of topological multiplicative recurrence.
\end{remark}

\subsection{Single recurrence and non-commutative semigroups}

Next theorem shows that in the special case $\ell=1$ of Theorem \ref{thm:multipledeterminant_intro}, one can drop the assumption that $f$ arises from a commutative semigroup. 


\begin{theorem}\label{thm_determinant}
	Let $(X, \mathcal{B}, \mu,T)$ be an $(\N, \times)$ measure preserving system and $A \in \mathcal{B}$ with $\mu(A)>0$. Let $f\colon (\N^k,\ast) \to \Q^{\geq 0}$ be a parametrized multiplicative function. Then
	\[
	    \liminf_{N\to\infty}\E_{n \in [N]^k}\mu(A\cap T_{f(n)}^{-1} A)>0.
	\]
\end{theorem}

\begin{proof}
Let $(\N^k,\ast)$ be the semigroup that gives the parametrized multiplicative function $f$. 
Let $Z$ be the set of all $n\in\N^{k}$ with $f(n)>0$. 
Since $f$ is a homomorphism, $(Z,\ast)$ is a subsemigroup of $(\N^{k},\ast)$.
For $n\in Z$, denote $R_{n}:=T_{f(n)}$.
	By \cref{thm:KhinchineGeneral}, the set
	$$S:=\big\{n\in Z:\mu(A\cap R^nA)>\mu^2(A)/2\big\}$$
	is syndetic in $(Z,\ast)$.
	Since $\mu(A\cap R^nA)\geq\tfrac12\mu(A)^2 1_S(n)$ for all $n\in Z$ and $\overline{d}_{+}(Z)=1$, it suffices to show that
	\begin{equation*}
	\liminf_{N\to\infty}\E_{n\in[N]^k\cap Z}1_S(n)>0.
	\end{equation*}
The last inequality follows from \cref{lemma_syndetic_liminf}, so our proof finishes. 
\end{proof}

\begin{remark}
    As an immediate corollary of Theorem \ref{thm_determinant}, we conclude that for all  measure preserving $(\N, \times)$-system $(X, \mathcal{B}, \mu,T)$ and $A \in \mathcal{B}$ with $\mu(A)>0$,
\[
    \liminf_{N\to\infty}\E_{n_{1},n_{2},n_{3},n_{4}\in[N]}\mu(A\cap T_{n_{1}^{2}+n_{2}^{2}+n_{3}^{2}+n_{4}^{2}}^{-1}A) > 0.
\]
\end{remark}

\section{Density regularity of Pythagorean triples in finite fields}
\label{sec:finite_Pythagorean}

\subsection{Pythagorean pairs in finite fields.}

In what follows we explore the question of finding Pythagorean pairs in sets of positive density in finite fields. We start with a result in number theory, which follows from \cite[Theorem 5A]{Schmidt_1976} (see also the proof of Corollary 5B there).

\begin{theorem}[{\cite[Theorem 5A]{Schmidt_1976}}]\label{thm_squaresrandom}
 
 Let $k\in\N$, let $F$ be a finite field, let $Q=\{x^2\colon x\in F\}$ and let $a_1,\dots,a_k\in F$ be arbitrary and pairwise distinct.
 Then
  $$\frac{\Big|(Q-a_1)\cap\cdots\cap(Q-a_k)\Big|}{|F|}=\frac1{2^k}+o_{k,|F|\to\infty}(1).$$
\end{theorem}

In next proposition, we show that \cref{ques:x^2-y^2} is true when $\N$ is replaced by a finite field. In fact, we prove a stronger result which involves ``density regularity'' rather than ``partition regularity''. 
\begin{proposition}
\label{thm:density_finite}
For every $\delta>0$ there is $M > 0$ such that whenever $\mathbb{F}$ is a finite field with $|\mathbb{F}|>M$ and $A\subset \mathbb{F}$ having $|A|>\delta|\mathbb{F}|$, there exist $x,y\in A$ with $x^2+y^2$ being a perfect square.
\end{proposition}
\begin{proof}
 Since $|A|>\delta|\mathbb{F}|$, the set $B:=\{x^2:x\in A\}$ has $|B|>\tfrac\delta2|\mathbb{F}|$.
 Our task is to show that there exist $x,y\in B$ such that $x+y$ is a perfect square when $|\mathbb{F}|$ is sufficiently large.
 Denote by $Q$ the set of perfect squares in $\mathbb{F}$; we will show that there exists $x\in B$ such that $(Q-x)\cap B\neq\emptyset$.
 This will follow if we show that the set $Q_B:=\bigcup_{x\in B}(Q-x)$ satisfies $|Q_B|>(1-\delta/2)|\mathbb{F}|$.
 
 Let $Q^c:=\mathbb{F}\setminus Q$ and $Q_B^c:=\mathbb{F}\setminus Q_B$.
 Then we have
 $Q_B^c=\bigcap_{x\in B}(Q^c-x)$ and we want to show that $|Q_B^c|<\tfrac\delta2|\mathbb{F}|$.
 Let $k\in\N$ be such that $1/2^k<\frac\delta2$.
 In view of \cref{thm_squaresrandom}, $|Q_B^c|=\frac1{2^k}+o_{k, |\mathbb{F}| \to\infty}(1)$, so if $|\mathbb{F}|$ is large enough we conclude that indeed $|Q_B^c|<\tfrac\delta2$, as desired.
 \end{proof}

\begin{remark} 
In contrast to \cref{thm:density_finite}, it is not true that every set of positive additive density of $\N$ contains $x,y$ such that $x^2 + y^2$ is a perfect square. For example, consider the set $4\N + 1$.
\end{remark}

\subsection{Pythagorean triples in finite fields}

It is also natural to ask whether an extension of \cref{thm:density_finite} stating that there exist $x, y, z \in A$ with $x^2 + y^2 = z^2$ still holds. 
We show that this is false and characterize all equations of the form $a x^2 + by^2 + cz^2 = 0$ for which the analogous results are true.

\begin{proposition}
  For $a,b,c\in\Z$, the followings are equivalent:
  \begin{enumerate}[label=(\roman*)]
    \item For every $\delta>0$, there is $M$ such that whenever $\mathbb{F}$ is a finite field with $\vert \mathbb{F} \vert>M$ and $A\subseteq \mathbb{F}$ with $\vert A\vert>\delta\vert \mathbb{F}\vert$, there exist $x,y,z\in A$ such that $ax^{2}+by^{2}+cz^{2}=0$.	
	
	\item $a + b + c = 0$.
  \end{enumerate}
\end{proposition}

\begin{remark}
    We remark that related questions are also raised and addressed in \cite{ Csikvari_Gyarmati_Sarkozy_2012,Lindqvist_2018}.
\end{remark}

\begin{proof}
Suppose that $a + b + c = 0$ and $\mathbb{F}$ is a finite field. Let $\delta > 0$ and $A \subset \mathbb{F}$ with $|A| > \delta |\mathbb{F}|$. Set $B:=\{x^{2}\colon x\in A\}$. Then $\vert B\vert \geq |A|/2 > \delta/2 |\mathbb{F}|$. It suffices to find $x,t\in \mathbb{F}$ such that $x,x+ct,x-bt\in B$, but this immediately follows from Szemer\'edi Theorem for finite groups (\cite[Theorem 10.5]{Tao_Vu06}).

Now assume that $a+b+c\neq 0$ and let $p >2$ be a prime. 
Let  $m=\vert a+b+c\vert>0$ and $q$ be the smallest prime such that $q > \max\{m,5\}$. 
For $0\leq i\leq m-1$ and $0\leq j\leq q-1$, denote 
\[
    B_{i,j}=\Big\{x\in\Z\colon iN/m\leq x<(i+1)p/m, x\equiv j \pmod q \text{ and $x$ is a perfect square} \pmod p\Big\}
\]
and let $A_{i,j}=\{x\in \mathbb{F} \colon x^{2}\in B_{i,j}\}$. 
Then
$\sum_{i=0}^{m-1}\sum_{j=0}^{p-1}\vert A_{i,j}\vert=p$.

We say that $(i,j)$ is \emph{good} if 
$mj\neq ip \mod q$. Let $U$ be the set of $(i,j), 0\leq i\leq m-1,0\leq j\leq q-1$ which is good, and $V$ be the set of  $(i,j)$ which is not good.
Using the estimation $\vert A_{i,j}\vert\leq 2\vert B_{i,j}\vert\leq 2+2p/mq$ and the fact that $\vert V\vert\leq m$, we have that
$$\sum_{(i,j) \text{ is good}}\vert A_{i,j}\vert\geq p-2m-4p/q.$$
By the pigeonhole principle, if $p$ is sufficiently large, then there exists 
a good pair $(i,j)$ such that $\vert A_{i,j}\vert\geq (1-5/q)p/mq$. 
It suffices to show that $ax^{2}+by^{2}+cz^{2}\neq 0$ for all $x,y,z\in A_{i,j}$.

Indeed, for any $x,y,z\in B_{i,j}$, denote $x'=mx-ip$, $y'=my-ip$ and $z'=mz-ip$. Then $0\leq x',y',z'<p$. If $ax+by+cz=0 \mod p$, then $m(ax+by+cz)=0 \mod p$. By the construction of $x',y'$ and $z'$, we have that $ax'+by'+cz'=0 \mod p$. Assume that $ax+by+cz=kp$ and $ax'+by'+cz'=rp$ for some $k,r\in\Z$, then $(mk-r)p=(a+b+c)ip=mip$. In other words, $mk-r=mi$. This implies that $m\vert r$. On the other hand, since
$-mp<ax'+by'+cz'=rp<mp$, we must have that $r=0$ and $ax'+by'+cz'=0$. Therefore,
$$0=ax'+by'+cz'\equiv (mj-ip)(a+b+c) = \pm m(mj-ip) \mod q.$$
This is impossible since $(i,j)$ is a good pair.
This completes the proof.\end{proof}

\begin{remark}\label{example_multiplicativegeneric}
    An example due to Frantzikinakis shows that there exists a subset of $\N$ of positive multiplicative density that contains no triple $x, y, z$ such that $x^2 + y^2 = z^2$. The example is as follows: let $(\T \coloneqq \R/\Z, \mu, T)$ be the multiplicative  measure preserving system where $\mu$ is the Lebesgue measure and $T_n x = n^2 x$ for all $x \in \T$. Let $x_0$ be a generic point for $\mu$. Fix $\delta > 0$ small. Then the set $E = \{n \in \N: n^2 x_0 \mod 1 \in [1/2 - \delta, 1/2 + \delta]\}$ has multiplicative density $2 \delta > 0$. For $x, y, z \in E$, $(x^2 + y^2) x_0 \mod 1 \in [0, 2 \delta] \cup [1 - 2 \delta, 1)$ and $z^2 x_0 \in [1/2 - \delta, 1/2 + \delta]$. Hence $(x^2 + y^2) x_0 \neq z^2 x_0 \mod 1$. It follows that $x^2 + y^2 \neq z^2$.  
\end{remark}

 \section{Open questions}
 
In this section, we collect some open questions regarding sets of multiplicative recurrence that naturally arise from our study.
In \cite[Problem 6]{Frantzikinakis_Host_2017}, Frantzikinakis and Host ask whether the set $\{(m^2 + n^2)/m^2: m, n \in \N\}$ is a set of topological multiplicative recurrence. In the spirit of this paper, we can ask the average version of this question, and some other related configurations:
\begin{question}
    Let $(X, \mathcal{B}, \mu, T)$ be a multiplicative  measure preserving system and $A \in \mathcal{B}$ with $\mu(A) > 0$. Is it true that
    \[
        \liminf_{N \to \infty} \E_{n,m \in [N]} \mu(T_{n^2+m^2}^{-1} A \cap T_{m^2}^{-1} A) > 0,
    \]    
    or
    \[
        \liminf_{N \to \infty} \E_{n,m \in [N]} \mu(T_{n^2+n}^{-1} A \cap T_{m^2}^{-1} A) > 0?
    \]
\end{question}

\cref{thm:n+1_n_intro} gives criteria for the image of a M{\"o}bius transformation to be a set of topological multiplicative recurrence. However there are cases not covered by this theorem. For example, we do not know whether $\{(6n+3)/(6n+2): n \in \N\}$ is a set of topological multiplicative recurrence. Because of this reason, we ask: 

\begin{question}
    \label{ques:mobius_necessary_sufficient}
    For $a \in \N$, $b, d \in \Z$, is it true that $S = \{(an+b)/(an + d): n \in \N\}$ is a set of topological multiplicative recurrence if and only if $a|b$ or $a|d$?  
\end{question}

For a polynomial $P \in \Z[x]$, \cref{conj:polynomial_semigroup} shows that $\{P(n): n \in \N\}$ contains an infinite multiplicative semigroup if and only if $P(x) = (ax + b)^d$ for some $a, d \in \N$, $b \in \Z$ with $a | b(b-1)$ and $d \in \N$. Here we can ask:
\begin{question}
   For which polynomial $P \in \Z[x]$ is $\{P(n): n \in \N\}$ a set of topological multiplicative recurrence?
\end{question}

Even the answer for the following question is unknown:

\begin{question}
Is $\{n^2+1:n\in\N\}$ a set of topological  multiplicative recurrence?
\end{question}

It is known that the set of shifted primes $\P - 1 = \{p - 1: p \text{ is a prime}\}$ and $\P + 1 = \{p + 1: p \text{ is a prime}\}$ are sets of additive recurrence. Hence it is of interest to ask:

\begin{question}
    Are $\P - 1$ and $\P +1$ sets of topological multiplicative recurrence?
\end{question}

We remark that neither of the sets $\P-1$ and $\P+1$ contain a multiplicative semigroup, since for any $a\in\N$ we can factor $a^2-1=(a-1)(a+1)$ and $a^3+1=(a+1)(a^2-a+1)$ and hence $\P+1$ does not contain a perfect square and $\P-1$ does not contain a perfect cube.


\end{document}